\documentclass{amsart}

\usepackage{amsfonts, amssymb, amsmath, eucal, verbatim, amsthm, amscd, enumerate}
\usepackage{setspace}
\newtheorem{theorem}{Theorem}[section]

\newtheorem{lemma}[theorem]{Lemma}
\newtheorem{proposition}[theorem]{Proposition}

\theoremstyle{definition}

\theoremstyle{remark}
\newtheorem*{remark}{Remark}

\newtheorem*{acknowledgements}{Acknowledgements}

\numberwithin{equation}{section}
\newcommand{\diffp}{D_+^{\beta_{+}}}
\newcommand{\diffn}{D_-^{\beta_{-}}}
\newcommand{\lqr}{{L^q_tL^r_x}}

\author{Jonathan Bennett}
\address[Jonathan Bennett]{School of Mathematics, The Watson Building, University of Birmingham, Edgbaston,
Birmingham, B15 2TT, England}
\email{j.bennett@bham.ac.uk}

\author{Neal Bez}
\address[Neal Bez]{Department of Mathematics, Graduate School of Science and Engineering,
Saitama University, Saitama 338-8570, Japan}
\email[Corresponding author]{nealbez@mail.saitama-u.ac.jp}

\author{Susana Guti\'errez}
\address[Susana Guti\'errez]{School of Mathematics, The Watson Building, University of Birmingham, Edgbaston,
Birmingham, B15 2TT, England}
\email{s.gutierrez@bham.ac.uk}

\author{Sanghyuk Lee}
\address[Sanghyuk Lee]{Department of Mathematical Sciences, Seoul National University, Seoul 151-747, Korea}
\email{shklee@snu.ac.kr}


\begin{document}

\begin{abstract}
We establish smoothing estimates in the framework of hyperbolic Sobolev spaces for the velocity averaging operator $\rho$ of the solution of the kinetic transport equation.
If the velocity domain is either the unit sphere or the unit ball, then, for any exponents $q$ and $r$, we find a characterisation of the exponents $\beta_+$ and $\beta_-$, except possibly for an endpoint case, for which $D_+^{\beta_+}D_-^{\beta_-} \rho$ is bounded from space-velocity $L^2_{x,v}$ to space-time $L^q_tL^r_x$. Here, $D_+$ and $D_-$ are the classical
and hyperbolic derivative operators, respectively. In fact, we shall provide an argument which unifies these velocity domains and the velocity averaging estimates in either case are shown to be equivalent to mixed-norm bounds on the cone multiplier operator acting on $L^2$. We develop our ideas further in several ways, including estimates for initial data lying in certain Besov spaces, for which a key tool in the proof is the sharp $\ell^p$ decoupling theorem recently established by Bourgain and Demeter. We also show that the level of permissible smoothness increases significantly if we restrict attention to initial data which are radially symmetric in the spatial variable.
\end{abstract}

\date{\today}

\title[Estimates for the kinetic transport equation]{Estimates for the kinetic transport equation in hyperbolic Sobolev spaces}
%{Estimates for the kinetic transport equation in hyperbolic Sobolev spaces}
\maketitle %\thispagestyle{empty}

\section{Introduction}

In this paper we consider regularity estimates for velocity integrals of the solution
\[
F(x,v,t)=f(x-tv,v)
\]
of the kinetic transport equation
\[
(\partial _t +v\cdot\nabla)F=0, \qquad F(x,v,0)=f(x,v),
\]
where $(x,v,t) \in \mathbb{R}^d \times \mathbb{R}^d\times \mathbb{R}$. 
The regularising effect of velocity integration (or ``velocity averaging") of the form 
\begin{equation}\label{rhobasic}
\rho f(x,t)=\int_{\mathbb{R}^d}f(x-tv,v) \,\mathrm{d}\mu(v)
\end{equation}
for various velocity measures $\mu$ has received considerable attention in the literature, where they are often referred to as \emph{velocity averaging lemmas} (see, for example, \cite{Bezard}, \cite{Bouchut2002}, \cite{BDes}, \cite{DeVP}, \cite{DLM}, \cite{Gerard}, \cite{GLPS}, \cite{GPR}, \cite{JabinVegaParis}, \cite{JabinVegaJMPA}, \cite{Lions}, \cite{LionsII}). Inequalities of this type are extremely rich, capturing diverse phenomena from geometric and harmonic analysis. This is perhaps most apparent through the interpretation of the dual operation
\begin{equation}\label{rhod}
\rho^*g(x,v) = \int_{\mathbb{R}}g(x+tv,t) \, \mathrm{d}t
\end{equation}
as a (space-time) $X$-ray transform, for which important problems remain wide open; see, for example, \cite{LT} or \cite{WolffX}. 
%for a formulation of the longstanding Kakeya maximal conjecture in these terms. %In addition to this manifest connection with combinatorial geometry, the velocity average $\rho f$ is naturally studied via the Fourier transform, allowing one to draw on methods and perspectives from classical harmonic analysis (examples or citations here). 

For the purposes of this introductory section, we focus our attention on the (physically-relevant) velocity average
\[
\rho f(x,t) = \int_{\mathbb{S}^{d-1}} f(x-tv,v) \, \mathrm{d}\sigma(v),
\]
where $\sigma$ is the induced Lebesgue measure on the unit sphere $\mathbb{S}^{d-1}$. Our estimates will capture a natural regularising effect of the averaging operator $\rho$ through the use of hyperbolic Sobolev spaces, and we begin by introducing our results in the context of initial data in $L^2(\mathbb{R}^d \times \mathbb{S}^{d-1})$. For example, given any $q,r \in [2,\infty)$, we shall obtain the optimal range of exponents $\beta_+$ and $\beta_-$ (except possibly an endpoint case) for which the global space-time estimate
\begin{equation} \label{e:sobs}
\| \diffp\diffn  \rho f\|_{L^q_tL_x^r}\le C\|f\|_{L^2_{x,v}}
\end{equation}
holds. Here, $D_+^{\beta_+}$ denotes classical fractional differentiation of order $\beta_+$ and $D_-^{\beta_-}$ denotes the hyperbolic differentiation operator of order $\beta_-$; these are Fourier multiplier operators with multipliers $(|\xi| + |\tau|)^{\beta_+}$ and $||\xi| - |\tau||^{\beta_-}$, respectively.

As far as we are aware, Bournaveas and Perthame \cite{BP} were the first to investigate regularising properties of velocity averages over spheres using hyperbolic Sobolev spaces. They obtained \eqref{e:sobs} in the case $(q,r) = (2,2)$ when $(d,\beta_+,\beta_-) = (3,\frac{1}{2},0)$ and $(d,\beta_+,\beta_-) = (2,\frac{1}{4},\frac{1}{4})$. Notice that in the two-dimensional case, a \emph{total} of $\frac{1}{2}$-derivative has been gained by the velocity average through the inclusion of hyperbolic derivatives; it was observed in \cite{BP} that such a gain is not possible by considering classical derivatives alone. These results were extended to all space dimensions $d \geq 2$ in \cite{BGRevista} and it was shown that \eqref{e:sobs} holds whenever $(q,r) = (2,2)$ and $(\beta_+,\beta_-) = (\frac{d-1}{4},-\frac{d-3}{4})$.

We now state our first main result which gives an extension of these results to $q,r \in [2,\infty)$. In general the total number of derivatives is given by
\begin{equation} \label{e:scaling}
\beta_+ + \beta_- = \frac{d}{r}+\frac{1}{q} -\frac{d}{2}.
\end{equation}
This restriction is in fact a necessary condition for \eqref{e:sobs} to hold, as can be shown by a simple scaling argument. Also, it will be useful to write
\[
\beta^*_+ = \min\bigg\{ \frac{d+1}{2r} - \frac{1}{2}, \frac{d}{r} + \frac{1}{q} - \frac{d+1}{4} \bigg\}.
\]
\begin{theorem}\label{t:thm-sobs} Let $d \geq 2$, $q,r\in [2,\infty)$ and suppose $\beta_+$, $\beta_-$ satisfy \eqref{e:scaling}. 
\begin{enumerate}
\item Suppose $\frac{1}{q} \leq \frac{d-1}{2}(\frac{1}{2} - \frac{1}{r})$. Then \eqref{e:sobs} holds if and only if $\beta_+ < \beta^*_+$.
\item Suppose $\frac{1}{q} > \frac{d-1}{2}(\frac{1}{2} - \frac{1}{r})$. Then \eqref{e:sobs} holds if $\beta_+ < \beta^*_+$ and fails if $\beta_+ > \beta^*_+$.
\end{enumerate}
\end{theorem}
The statement is given in terms of the parameter $\beta_+$, providing the upper threshold on the number of allowable classical derivatives. This in the spirit of the genesis of such estimates, however, $\beta_-$ may be considered the more decisive parameter since its lower threshold is negative and thus a singularity in the $D_-$ multiplier appears. Thus, we shall also write
\[
\beta^*_- = \max\bigg\{ \frac{1}{q} + \frac{d-1}{2r} - \frac{d-1}{2}, -\frac{d-1}{4}\bigg\}
\]
for the lower threshold in $\beta_-$ so that $\beta_+ < \beta_+^*$ if and only if $\beta_- > \beta_-^*$.

We will give two different proofs of the sufficiency claims in Theorem \ref{t:thm-sobs}, one of which relies on duality and Plancherel's theorem, and another which proceeds by a direct analysis of the operator $\rho$. The dual approach is special to the case of initial data in $L^2(\mathbb{R}^d \times \mathbb{S}^{d-1})$. Nevertheless, it allows us to highlight a strikingly clear connection to the cone multiplier operator, a well-known operator in harmonic analysis, whose \emph{full} range of bounds on Lebesgue spaces is a famous open problem. For $\alpha  > -1$, the cone multiplier operator $\mathcal{C}^\alpha$ of order $\alpha$ will be given by
\[
\widehat{\mathcal{C}^\alpha g}(\xi,\tau) = \bigg(1-\frac{\tau^2}{|\xi|^2}\bigg)^\alpha_+\phi(|\xi|) \widehat{g}(\xi,\tau)
\]
where $\phi \in C^\infty_c(\mathbb{R})$ is supported in $[\frac{1}{2},2]$. We note that the conventional cone multiplier operator (first introduced in \cite{SteinAMS}) is given by the multiplier $(1-|\xi|^2/\tau^2)^\alpha_+\phi(\tau)$, however we may consider these operators as essentially the same with regard to their boundedness properties and thus we continue to refer to $\mathcal{C}^\alpha$ as the cone multiplier of order $\alpha$.
\begin{theorem} \label{t:coneequiv}
Let $d \geq 2$, $q,r\in [2,\infty)$ and suppose $\beta_+$, $\beta_-$ satisfy \eqref{e:scaling}. Then the estimate \eqref{e:sobs} holds if and only if $\mathcal{C}^{\beta_- + \frac{d-3}{4}}$ is $L^2_{t,x} \to L^q_tL^r_x$ bounded.
\end{theorem}
Thus, $\mathcal{C}^\alpha$ is the fundamental operator whose mixed-norm bounds for functions in $L^2_{t,x}$ underpin the smoothing estimates \eqref{e:sobs}. Even in the case $(q,r) = (2,2)$, this gives a new perspective by showing that the bounds established in \cite{BP} and \cite{BGRevista} at $(\beta_+,\beta_-) = (\frac{d-1}{4},\frac{3-d}{4})$ are equivalent to the (elementary) boundedness of $\mathcal{C}^0$ on $L^2$.

Naturally, we would like to establish a full understanding of the mixed-norm estimates of $\mathcal{C}^\alpha$ for functions in $L^2$. Although the $L^p \to L^q$ bounds for this operator have been extensively studied (see, for example, \cite{Bourgaincone}, \cite{GarrigosSeeger}, \cite{Heo}, \cite{HeoNazarovSeegerI}, \cite{HeoNazarovSeegerII}, \cite{LabaWolff}, \cite{Lee}, \cite{LeeVargas}, \cite{Mockenhaupt}, \cite{TaoVargasI}, \cite{TaoVargasII}, \cite{Wolff}) we were not able to find a reference for the mixed-norm estimates that we need in the present work and thus, in Section \ref{section:conethm}, we shall include a proof of the following result. Moreover, the argument given there is the basis of our direct approach to proving Theorem \ref{t:thm-sobs}.

Let
\begin{equation} \label{e:alphastar}
\alpha^* = \alpha^*(q,r) = \max\bigg\{\frac{1}{q}+\frac{d-1}{2r}-\frac{d+1}{4}, -\frac{1}{2}\bigg\}.
\end{equation}
\begin{theorem} \label{t:coneestimates}
Let $d \geq 2$ and $q,r\in [2,\infty)$.
\begin{enumerate}
\item Suppose $\frac{1}{q} \leq \frac{d-1}{2}(\frac{1}{2} - \frac{1}{r})$. Then $\mathcal{C}^{\alpha}$ is $L^2_{t,x} \to L^q_tL^r_x$ bounded if and only if $\alpha > \alpha^*$.
\item Suppose $\frac{1}{q} > \frac{d-1}{2}(\frac{1}{2} - \frac{1}{r})$. Then $\mathcal{C}^{\alpha}$ is $L^2_{t,x} \to L^q_tL^r_x$ bounded if $\alpha > \alpha^*$ and unbounded if $\alpha < \alpha^*$.
\end{enumerate}
\end{theorem}
We may say that $(q,r)$ is \textit{wave-admissible} when $\frac{1}{q} \leq \frac{d-1}{2}(\frac{1}{2} - \frac{1}{r})$ and $q,r \in [2,\infty)$. This is common terminology and, since we consider the case $r < \infty$, wave-admissibility is equivalent to the validity of the classical Strichartz estimates $\dot{H}^s \times \dot{H}^{s-1} \to L^q_tL^r_x$ for the solution of the wave equation $(\partial_t^2 - \Delta)u = 0$, where $s = \frac{d}{2} - \frac{d}{r} - \frac{1}{q}$. These estimates form the basis for our proofs of Theorem \ref{t:thm-sobs} (via the direct approach and the dual approach), along with some additional arguments when $d=2,3$, and classical Littlewood--Paley theory.

Our direct approach to proving the sufficiency claims in Theorem \ref{t:thm-sobs} has the merit that it naturally extends beyond the case where the initial data lies in $L^2(\mathbb{R}^d \times \mathbb{S}^{d-1})$. For example, we shall use this approach to establish an extension of Theorem \ref{t:thm-sobs} for initial data in certain Besov spaces making use of the sharp $\ell^p$ decoupling inequality for the cone recently established by Bourgain and Demeter \cite{BD}. In a different direction of development, we shall see that the direct approach allows us to see an additional gain of regularity if we restrict to initial data which are radially symmetric in the spatial variable. This argument uses the Funk--Hecke theorem, a result from classical harmonic analysis, and permits data which are rougher than $L^2(\mathbb{R}^d \times \mathbb{S}^{d-1})$, with regularity measured with respect to smoothing in the spherical variable.

Both approaches readily allow us to understand velocity averages over different sets $V$. The case where $V$ is the closed unit ball $\mathbb{B}^{d-1}$ (equipped with the Lebesgue measure) has also featured prominently in the literature on velocity averages. As will become clear in Section \ref{section:dual}, our dual approach will be used to see that the analogous estimate to \eqref{e:sobs} on $\mathbb{B}^{d-1}$ is equivalent to the $L^2_{t,x} \to L^q_tL^r_x$ boundedness of $\mathcal{C}^{\beta_-+\frac{d-1}{4}}$ (and thus the analogue of Theorem \ref{t:thm-sobs} holds with $\beta^*_+$ raised by $\frac{1}{2}$). In this sense, we can view the cases of the sphere and the unit ball as equivalent with a unified treatment. Our direct approach may also be used to obtain bounds over more general velocity domains and we illuminate this point at the end of Section \ref{subsection:directproof}.

The present work is a contribution to a large body of work on velocity averages in the context of kinetic equations. The papers \cite{BP} and \cite{BGRevista} already cited above have the most direct connection to our work. For comprehensive accounts of the original motivation for studying regularising properties of velocity averages, along with extensive summaries of the prior results, we refer the reader to the excellent surveys of Bouchut \cite{Bouchut} and Perthame \cite{PerthameBAMS}.

\section{Overview and organisation}
We have intentionally stated only a sample of our results in the introductory section. The current section provides a more detailed overview of our main contributions and allows us to clarify the structure.

\subsection*{Section \ref{section:pre}} We establish some notation and record some observations which will be used frequently throughout the paper. 

\subsection*{Section \ref{section:dual}} We present our dual approach to smoothing estimates for $\rho$ with $L^2$ initial data, beginning by allowing for the velocity domain and multiplier to be non-specific, then establishing a duality principle between such estimates and certain Fourier multiplier estimates. This culminates with a proof of a generalisation of Theorem \ref{t:coneequiv} which includes both $\mathbb{S}^{d-1}$ and $\mathbb{B}^{d-1}$ as the velocity domain as special cases; see Theorem \ref{t:genconeequiv}.

\subsection*{Section \ref{section:conethm}} We prove Theorem \ref{t:coneestimates} based on the duality principle from Section \ref{section:dual}, thus giving a proof of Theorem \ref{t:thm-sobs}. Using Theorem \ref{t:genconeequiv}, we shall in fact simultaneously give a proof of the analogue of Theorem \ref{t:thm-sobs} for $\mathbb{B}^{d-1}$; see Theorem \ref{t:thm-sobsgen}. We also establish certain weak type estimates for $\mathcal{C}^\alpha$ in the critical case $\alpha = \alpha^*$.

\subsection*{Section \ref{section:direct}} We present our direct approach to smoothing estimates for $\rho$. The natural setting for the argument is for initial data in an $L^p$-based Besov space and initially we illustrate how such estimates crucially depend on $L^p_{x,v} \to L^q_tL^r_x$ bounds for $\mathcal{C}_k \circ \rho$, where $\mathcal{C}_k$ is a Fourier multiplier operator supported in a $2^{-k}$ neighbourhood of the truncated cone. In particular, the range of $\beta_-$ is completely determined by the bound on this operator. When $p=2$ our argument leads to a direct proof of Theorem \ref{t:thm-sobs}, and for general $p \geq 2$, based on the sharp $\ell^p$ decoupling theorem of Bourgain and Demeter \cite{BD}, we establish smoothing estimates for initial data in the Besov space $\dot{B}^s_{p,2}$ (see Theorem \ref{t:Besov}).

\subsection*{Section \ref{section:furthers}} We present several further results and contextual remarks. As is clear from our discussion to this point, a feature of this paper is the exposing of links with contemporary aspects of harmonic analysis, and in particular the modern theory of Fourier multipliers and the restriction theory of the Fourier transform. Additional discussion along these lines appears in Section \ref{section:furthers}, where, for example, the affine-invariant endpoint multilinear Kakeya inequality (see \cite{BG} and \cite{CV}) is viewed as a null-form estimate for a certain multilinear variant of $\rho$. 

In a different direction, we use our direct approach in Section \ref{section:direct} to identify an improving effect obtainable by restricting to initial data which are assumed to possess some symmetry; in particular, radially symmetric with respect to the spatial variable, and independence with respect to the velocity variable. We are also able to use our duality principle to identify the optimal constant and fully address the existence and characterisation of extremisers for the estimate \eqref{e:sobs} whenever $(q,r) = (2,2)$ (see Theorem \ref{t:sharpgeneral}). Optimal constants and extremisers are also identified when restricting to initial data which are radial in the spatial variable (see Theorem \ref{t:sharpradial}).

\section{Notation and preliminaries} \label{section:pre}

\subsection{Notation} For space-time functions defined on $\mathbb{R}^{d+1}$ we consistently use the letter $g$, and for space-velocity functions defined on $\mathbb{R}^d \times V$, we consistently use the letter $f$. If a function is defined on $\mathbb{R}^d$ we use the letter $h$.

From now on, in the Lebesgue space notation, we shall usually drop the explicit reference to the underlying measure space (for example, $L^2_{x,v}$ will simply be written $L^2$ where possible). All norms are global and so there should be no confusion.

Regarding constants, we write $A \lesssim B$ or $B \gtrsim A$ to mean $A \leq C B$, where $C$ is a constant which is allowed to depend on $d$ and any exponents which are used to define the relevant function space in use, and $A \sim B$ means both $A \lesssim B$ and $B \lesssim A$ hold. %Also, we will write $A \approx B$ if $A = C B$, where $C$ is a constant which is allowed to depend on $d$ and any exponents which are used to define the relevant function space in use.

We introduce the Littlewood--Paley projection operators $(P_j)_{j \in \mathbb{Z}}$ given by
\[
\widehat{P_j h}(\xi) = \varphi(2^{-j}|\xi|)\widehat{h}(\xi)
\]
where $\varphi \in C^{\infty}_c(\mathbb{R})$ is supported in $[\frac{1}{2},2]$ and such that
\[
\sum_{j \in \mathbb{Z}} \varphi(2^{-j}|\xi|) = 1
\]
for all $\xi \neq 0$, and for appropriate functions $h$ on $\mathbb{R}^d$. We extend the definition of $P_j$ to space-time functions on $\mathbb{R}^{d} \times \mathbb{R}$ or  space-velocity functions on $\mathbb{R}^d \times V$ in the obvious manner acting on the spatial variable only.
%\[ \widehat{P_j f}(\xi,v) = \varphi(2^{-j}|\xi|)\widehat{f}(\xi,v)\]
The classical Littlewood--Paley inequality that we will use states that, for any $r \in (1,\infty)$ we have the equivalence
\begin{equation*} \label{e:LP}
\|h\|_{L^r(\mathbb{R}^d)} \sim \bigg\| \bigg( \sum_{j \in \mathbb{Z}} |P_jh|^2 \bigg)^{1/2} \bigg\|_{L^r(\mathbb{R}^d)}.
\end{equation*}
For a proof and further discussion we refer the reader to \cite{Stein}.

The annulus $\{ \xi \in \mathbb{R}^d : |\xi| \in [\frac{1}{2},2]\}$ will be denoted by $\mathfrak{A}_0$ and the indicator function of the set $S$ will be written as $\mathbf{1}_S$, and $x_+ = \max\{x,0\}$. Also, the Fourier transform of an integrable function $\varphi : \mathbb{R}^n \to \mathbb{C}$ is defined to be
\[
\widehat{\varphi}(\xi) = \int_{\mathbb{R}^n} \varphi(x) e^{-ix\cdot \xi} \, \mathrm{d}x
\]
for $\xi \in \mathbb{R}^n$. Thus, we use the same $\,\,\widehat{\,\,}\,\,$ notation for functions depending on $x$, or $(x,t)$, or $(x,v)$. In the latter case of space-velocity functions, the meaning is that the Fourier transform is taken only with respect to the spatial variable. Also, sometimes it is convenient to write $\mathcal{F}\varphi = \widehat{\varphi}$.

\subsection{Preliminary observations}
For general velocity domains, the Fourier transform of $\rho f$ is easily computed and we obtain the expression
\begin{equation} \label{e:rhofhat}
\widehat{\rho f}(\xi,\tau) = 2\pi \int_{V} \delta(v \cdot \xi + \tau) \widehat{f}(\xi,v) \, \mathrm{d}\mu(v).
\end{equation}

Now suppose $V = \mathbb{S}^{d-1}$ with Lebesgue measure. Clearly $\widehat{\rho f}$ is supported in the region
\[
\mathfrak{C} := \{ (\xi,\tau) \in \mathbb{R}^{d+1} : |\tau| \leq |\xi| \}
\]
and this fact plays an important role in the analysis.

It will also be helpful to note here that the dual operator $\rho^*$ is given by
\begin{equation} \label{e:rhostar}
\rho^*g(x,v) = \int_{\mathbb{R}} g(x + tv,t)\, \mathrm{d}t
\end{equation}
and hence
\begin{equation} \label{e:rhostarhat}
\widehat{\rho^* g}(\xi,v) = \widehat{g}(\xi,-\xi \cdot v).
\end{equation}

The relevance of the multipliers $D_+$ and $D_-$ may be seen by considering initial data which are independent of the velocity variable. In this case, we may explicitly calculate the above integral in \eqref{e:rhofhat} over $\mathbb{S}^{d-1}$ using the following.
\begin{lemma} \label{l:multiplierroot}
For every $(\tau,\xi) \in \mathbb{R}^{d+1}$ with $\xi \neq 0$ we have
\[
\int_{\mathbb{S}^{d-1}} \delta(v \cdot \xi + \tau)  \, \mathrm{d}\sigma(v) = |\mathbb{S}^{d-2}| \frac{\mathbf{1}_\mathfrak{C}(\xi,\tau)}{|\xi|} \bigg(1-\frac{\tau^2}{|\xi|^2}\bigg)^{\frac{d-3}{2}}.
\]
\end{lemma}
\begin{proof}
We use rotation invariance and homogeneity to write
\[
\int_{\mathbb{S}^{d-1}} \delta(v \cdot \xi + \tau)  \, \mathrm{d}\sigma(v) = \frac{1}{|\xi|}\int_{\mathbb{S}^{d-1}} \delta(v \cdot e_d + \tfrac{\tau}{|\xi|})  \, \mathrm{d}\sigma(v)
\]
where $e_d$ is the $d$th standard basis vector in $\mathbb{R}^d$. Parametrising $\mathbb{S}^{d-1}$ by $v = (\tilde{v}\sin \theta, \cos \theta)$ for $\tilde{v} \in \mathbb{S}^{d-2}$ and $\theta \in [0,\pi]$ we obtain the claimed expression.
\end{proof}

Using Lemma \ref{l:multiplierroot} and rotation invariance, one may write
\begin{equation} \label{e:nodeltarep}
\widehat{\rho f}(\xi,\tau) = \frac{2\pi\mathbf{1}_\mathfrak{C}(\xi,\tau)}{(|\xi|^2 - \tau^2)^{1/2}} \int_{\Sigma_{\xi,\tau}} \widehat{f}(\xi,v) \, \mathrm{d}\sigma_{\xi,\tau}(v)
\end{equation}
where $\Sigma_{\xi,\tau} := \{ v \in \mathbb{S}^{d-1} : v \cdot \xi + \tau = 0\}$ is a slice of $\mathbb{S}^{d-1}$ by a hyperplane with normal direction given by $\xi$, and $\sigma_{\xi,\tau}$ is the induced surface measure. This alternative representation of $\widehat{\rho f}$ will sometimes be more convenient than \eqref{e:rhofhat}.

\section{Approach I : Dual analysis} \label{section:dual}

\subsection{A duality principle}
We begin by considering the velocity domain $V$ equipped with the measure $\mathrm{d}\mu(v) = w(v) \, \mathrm{d}v$ for some compactly supported function $w$ on $V$, and the corresponding velocity averaging operator
\[
\rho f(x,t) = \int_V f(x-tv,v) \, \mathrm{d}\mu(v).
\]
The duality principle that we would like to expose concerns smoothing estimates for $\rho$ of the form
\begin{equation} \label{e:generalsmoothing}
\| \mathcal{F}^{-1}(m \widehat{\rho f}) \|_{L^q_tL^r_x} \leq \mathbf{C} \| f \|_{L^2}
\end{equation}
for an appropriate Fourier multiplier $m$, and the $L^2 \to L^q_tL^r_x$ boundedness of the associated multiplier $m_\mu$ given by
\begin{equation} \label{e:mmudefn}
m_\mu(\xi,\tau) = m(\xi,\tau) \big(\tfrac{1}{|\xi|}\mathcal{R}w(-\tfrac{\tau}{|\xi|},\tfrac{\xi}{|\xi|}) \big)^{1/2}.
\end{equation}
The notation $\|f\|_{L^2}$ means that the integration in the $v$-variable is taken with respect to the measure $\mu$, and $\mathcal{R}$ is the Radon transform given by
\begin{equation*} \label{e:Radon}
\mathcal{R}w(r,\theta) = \int_{V} w(y) \delta(y \cdot \theta - r) \, \mathrm{d}y
\end{equation*}
averaging over the hyperplane $\{ y \in \mathbb{R}^d : y \cdot \theta = r\}$ for fixed $(r,\theta) \in \mathbb{R} \times \mathbb{S}^{d-1}$. Also, we use the boldface notation $\mathbf{C}$ for the optimal constant in \eqref{e:generalsmoothing}.

First, note that \eqref{e:generalsmoothing} is dual to the estimate
\begin{equation*}
\| \rho^* \mathcal{F}^{-1}(m\widehat{g}) \|_{L^2} \leq \mathbf{C} \|g\|_{L^{q'}_t L^{r'}_x}
\end{equation*}
where the dual operator $\rho^*$ is given by \eqref{e:rhostar}. Using \eqref{e:rhostarhat} we easily obtain
\begin{align*}
\| \rho^* \mathcal{F}^{-1}(m\widehat{g}) \|_{L^2}^2 & = \frac{1}{(2\pi)^d} \int_{\mathbb{R}^d} \int_{\mathbb{R}}  |m(\xi,\tau)|^2 |\widehat{g}(\xi,\tau)|^2 \int_V \delta(\tau + v \cdot \xi) w(v) \, \mathrm{d}v \mathrm{d}\tau\mathrm{d}\xi
\end{align*}
and hence
\begin{equation*} \label{e:mainpoint}
\| \rho^* \mathcal{F}^{-1}(m\widehat{g}) \|_{L^2}^2 = \frac{1}{(2\pi)^{d}}\| m_\mu \widehat{g} \|_{L^2}^2 = 2\pi \| \mathcal{F}^{-1}(m_\mu\widehat{g})\|_{L^2}^2.
\end{equation*}
This means that \eqref{e:generalsmoothing} is equivalent to the $L^{q'}_t L^{r'}_x \to L^2$ boundedness of the Fourier multiplier $m_\mu$, and by a further duality, we have proved the following.
\begin{theorem}[Duality Principle] \label{t:dualityprinciple}
The smoothing estimate
\begin{equation*}
\| \mathcal{F}^{-1}(m \widehat{\rho f}) \|_{L^q_tL^r_x} \leq \mathbf{C} \| f \|_{L^2}
\end{equation*}
holds if and only if $m_\mu$ is a bounded Fourier multiplier from $L^2$ to $L^q_tL^r_x$. Moreover, the optimal constant $\mathbf{C}$ is such that $(2\pi)^{-1/2}\mathbf{C}$ coincides with the operator norm of the Fourier multiplier associated with $m_\mu$ as a mapping from $L^2$ to $L^q_tL^r_x$. 
\end{theorem}
A particular instance of this duality principle can be found in work of Bouchut \cite{Bouchut} (see Proposition 7.1) corresponding to the case where $(q,r) = (2,2)$ and classical Sobolev smoothing.

In the case of a radially symmetric measure supported inside the unit ball $\mathbb{B}^{d-1}$, the following expression for the Radon transform will be convenient.
\begin{proposition} \label{p:radialRadon}
If $w(v) = \tilde{w}(|v|)\mathbf{1}_{[0,1]}(|v|)$, then
\[
\mathcal{R}w(r,\theta) = |\mathbb{S}^{d-2}| \mathbf{1}_{[-1,1]}(r) \int_{|r|}^1 \tilde{w}(s) s^{d-2} (1 - \tfrac{r^2}{s^2})^{\frac{d-3}{2}} \, \mathrm{d}s
\]
for each $(r,\theta) \in \mathbb{R} \times \mathbb{S}^{d-1}$.
\end{proposition}
\begin{proof}
For $r \in [0,1]$, using polar coordinates we get
\begin{align*}
\mathcal{R}w(r,\theta) = |\mathbb{S}^{d-2}| \int_0^1 \int_{-1}^1 \tilde{w}(s) \delta(s\lambda - r)(1-\lambda^2)^{\frac{d-3}{2}}s^{d-1} \, \mathrm{d}\lambda \mathrm{d}s
\end{align*}
and the claimed expression follows.
\end{proof}

The cases of primary interest are $V = \mathbb{S}^{d-1}$ or $V = \mathbb{B}^{d-1}$ equipped with the induced Lebesgue measures and
\[
m(\xi,\tau) =  (|\xi| + |\tau|)^{\beta_+}  | |\xi| - |\tau| |^{\beta_-}.
\]
These cases may be unified by considering the case where
\begin{equation} \label{e:wkappa}
w_\kappa(v) = \frac{1}{\Gamma(1+\kappa)} (1-|v|^2)^{\kappa} \mathbf{1}_{[0,1]}(|v|)
\end{equation}
for $\kappa \in [-1,0]$; we let $\mathrm{d}\mu_\kappa(v) = w_\kappa(v) \, \mathrm{d}v$ and, to avoid double subscripts, we write $m_\kappa$ for the associated multiplier given by \eqref{e:mmudefn}. This family of measures naturally unifies the cases of interest since $\mu_{-1} = \frac{1}{2}\sigma(v)$ and $\mathrm{d}\mu_0(v) = \mathbf{1}_{\mathbb{B}^{d-1}}(v)\,\mathrm{d}v$; of course, $\mu_{-1}$ is a singular measure and so this ceases to directly fall under the scope of the above discussion; however a limiting argument allows us to make sense of $\mathcal{R}w_{-1}$ and we obtain
\[
\mathcal{R}w_{-1}(r,\theta) = \tfrac{1}{2}|\mathbb{S}^{d-2}|  (1 - r^2)_+^{\frac{d-3}{2}}
\]
and therefore
\begin{equation} \label{e:m-1}
m_{-1}(\xi,\tau) = (\tfrac{1}{2}|\mathbb{S}^{d-2}|)^{1/2} |\xi|^{\beta_+ + \beta_- - \frac{1}{2}}  \bigg(1 +  \frac{|\tau|}{|\xi|}\bigg)^{\beta_+ - \beta_-} \bigg(1 - \frac{\tau^2}{|\xi|^2}\bigg)_+^{\beta_- + \frac{d-3}{4}}.
\end{equation}
\begin{remark}
The limiting argument we refer to in order to obtain the above formula for $\mathcal{R}w_{-1}$ is already present in Lemma \ref{l:multiplierroot}. The operator $\mathcal{R}$ acting on more general singular measures can be shown to be well defined and we refer the reader to \cite{Mattila} for further details on such sliced measures. 
\end{remark}

For general $\kappa \in [-1,0]$ we use Proposition \ref{p:radialRadon} followed by elementary changes of variables to write
\begin{align*}
\mathcal{R}w_\kappa(r,\theta) & = \mathbf{1}_{[-1,1]}(r)\frac{|\mathbb{S}^{d-2}|}{\Gamma(1 + \kappa)} \int_{|r|}^1 (1-s^2)^\kappa (s^2 - r^2)^{\frac{d-3}{2}} \, s\mathrm{d}s \\
& = \frac{|\mathbb{S}^{d-2}|}{2\Gamma(1 + \kappa)} (1 - r^2)_+^{\kappa + \frac{d-1}{2}} \mathrm{B}(1 + \kappa, \tfrac{d-1}{2}).
\end{align*}
The beta function $\mathrm{B}$ satisfies the identity $\Gamma(x+y) \mathrm{B}(x,y) = \Gamma(x)\Gamma(y)$, hence
\begin{equation*}
\mathcal{R}w_\kappa(r,\theta) = \frac{|\mathbb{S}^{d-2}|\Gamma(\frac{d-1}{2})}{2\Gamma(\frac{d+1}{2} + \kappa)} (1-r^2)_+^{\kappa + \frac{d-1}{2}}
\end{equation*}
and whence
\begin{equation} \label{e:mgamma}
m_\kappa(\xi,\tau) = \bigg(\frac{|\mathbb{S}^{d-2}|\Gamma(\frac{d-1}{2})}{2\Gamma(\frac{d+1}{2} + \kappa)}\bigg)^{1/2} |\xi|^{\beta_+ + \beta_- - \frac{1}{2}}  \bigg(1 +  \frac{|\tau|}{|\xi|}\bigg)^{\beta_+ - \beta_-} \bigg(1 - \frac{\tau^2}{|\xi|^2}\bigg)_+^{\alpha}
\end{equation}
where $\alpha = \beta_- + \frac{\kappa}{2} + \frac{d-1}{4}$.

\begin{remark}
Observe that we are now in a position to immediately give a clear picture of when \eqref{e:sobs} holds in the case $(q,r) = (2,2)$. Indeed, by Theorem \ref{t:dualityprinciple}, \eqref{e:sobs} holds if and only if $m_{-1}$ is a bounded multiplier $L^2 \to L^2$, that is to say, $m_{-1}$ is a bounded function on $\mathbb{R}^{d+1}$. From \eqref{e:m-1}, clearly this is the case if and only if $\beta_+ + \beta_- = \frac{1}{2}$ and $\beta_- + \frac{d-3}{4} \geq 0$. This approach based on the duality principle provides an alternative to that given in \cite{BP} and \cite{BGRevista}. Of course, \eqref{e:m-1} also brings to light the link to the cone multiplier $\mathcal{C}^{\beta_- + \frac{d-3}{4}}$ and this forms the basis, along with Theorem \ref{t:dualityprinciple}, for our proof of Theorem \ref{t:coneequiv}.
\end{remark}
More generally, we prove the following for the velocity average $\rho_\kappa$ given by
\[
\rho_\kappa f(x,t) = \int_{\mathbb{R}^d} f(x-tv,v)\,\mathrm{d}\mu_\kappa(v).
\]
\begin{theorem} \label{t:genconeequiv}
Let $d \geq 2$, $q,r\in [2,\infty)$, $\kappa \in [-1,0]$ and suppose $\beta_+$, $\beta_-$ satisfy \eqref{e:scaling}. Then the estimate
\begin{equation} \label{e:sobsgen}
\| \diffp \diffn  \rho_\kappa f\|_{L^q_tL_x^r}\lesssim \|f\|_{L^2}
\end{equation}
holds if and only if $\mathcal{C}^{\alpha}$ is $L^2 \to L^q_tL^r_x$ bounded, where $\alpha = \beta_- + \frac{\kappa}{2} + \frac{d-1}{4}$.
\end{theorem}
\begin{proof}
From Theorem \ref{t:dualityprinciple}, it suffices to show that the $L^2 \to L^q_tL^r_x$ boundedness of the Fourier multiplier $m_\kappa$ and $\mathcal{C}^\alpha$ are equivalent. Obviously
$
\mathcal{F}(\mathcal{C}^\alpha g) = m_\kappa \tilde{m} \widehat{g}
$
where
\[
\tilde{m}(\xi,\tau) = C_{d,\kappa}\mathbf{1}_\mathfrak{C}(\xi,\tau)\phi(|\xi|) |\xi|^{-\beta_+ - \beta_- + \frac{1}{2}} \bigg(1 + \frac{|\tau|}{|\xi|}\bigg)^{\beta_- - \beta_+}
\]
and $C_{d,\kappa}$ is some constant. It follows that $\mathcal{F}^{-1} \tilde{m} \in L^1$ and hence the $L^2 \to L^q_tL^r_x$ boundedness of $\mathcal{C}^\alpha$ follows from the $L^2 \to L^q_tL^r_x$ boundedness of the Fourier multiplier $m_\kappa$.

Conversely, if we assume that $\mathcal{C}^\alpha$ is $L^2 \to L^q_tL^r_x$ bounded, a similar argument shows that
\begin{equation} \label{e:mgamma0}
\| \mathcal{F}^{-1} (m_\kappa \widehat{P_0g})\|_{L^q_tL^r_x} \lesssim \|P_0g\|_{L^2}
\end{equation}
where, in general, $\widehat{P_j g}(\xi,\tau) = \phi(2^{-j}|\xi|)\widehat{g}(\xi,\tau)$ is the $j$th Littlewood--Paley projection operator. Since
\[
\mathcal{F}^{-1}(m_\kappa \widehat{P_jg})(x,t) = 2^{(\beta_+ + \beta_- + d + \frac{1}{2})j} \mathcal{F}^{-1}(m_\kappa \widehat{P_0g_j})(2^jx2^jt)
\]
where $\widehat{g_j}(\xi,\tau) = \widehat{g}(2^j\xi,2^j\tau)$, we see that \eqref{e:mgamma0} implies
\begin{equation*}
\| \mathcal{F}^{-1} (m_\kappa \widehat{P_jg})\|_{L^q_tL^r_x} \lesssim \|P_jg\|_{L^2}
\end{equation*}
for all $j \in \mathbb{Z}$. Since $q,r \in [2,\infty)$, it now follows from \eqref{e:LP} that the Fourier multiplier $m_\kappa$ is bounded $L^2 \to L^q_tL^r_x$.
\end{proof}

As a consequence of Theorem \ref{t:coneestimates} (to be proved in the forthcoming section), we obtain the following generalisation of Theorem \ref{t:thm-sobs} given in terms of the threshold
\[
\beta^*_-(\kappa) = \max\bigg\{ \frac{1}{q} + \frac{d-1}{2r} - \frac{d + \kappa}{2}, -\frac{d + 1 + 2\kappa}{4} \bigg\}.
\]
\begin{theorem}\label{t:thm-sobsgen} Let $d \geq 2$, $q,r\in [2,\infty)$, $\kappa \in [-1,0]$ and suppose $\beta_+$, $\beta_-$ satisfy \eqref{e:scaling}.
\begin{enumerate}
\item Suppose $\frac{1}{q} \leq \frac{d-1}{2}(\frac{1}{2} - \frac{1}{r})$. Then \eqref{e:sobsgen} holds if and only if $\beta_- > \beta^*_-(\kappa)$.
\item Suppose $\frac{1}{q} > \frac{d-1}{2}(\frac{1}{2} - \frac{1}{r})$. Then \eqref{e:sobsgen} holds if $\beta_- > \beta^*_-(\kappa)$ and fails if $\beta_- < \beta^*_-(\kappa)$.
\end{enumerate}
\end{theorem}

\section{Proof of Theorem \ref{t:coneestimates}} \label{section:conethm}

\subsection{Sufficiency}
Fix $\alpha > \alpha^*$. The localisation to $(\tau,\xi) \in \mathfrak{C}$ with $\xi \in \mathfrak{A}_0$ built into the operator $\mathcal{C}^\alpha$ means that the desired estimate
\[
\|\mathcal{C}^\alpha g\|_{L^q_tL^r_x} \lesssim \|g\|_{L^2}
\]
follows once we prove that
$
\phi(|\xi|)(|\xi| - |\tau|)_+^\alpha
$
gives rise to a bounded multiplier operator $L^2 \to L^q_tL^r_x$. Indeed, since $|\xi| \sim 1$ and $|\tau| \lesssim 1$, elementary considerations show that the convolution kernel corresponding to the remaining factor in the multiplier is integrable on $\mathbb{R}^d \times \mathbb{R}$ and thus plays a benign role.

Since $\alpha$ may be negative, the delicate part of the multiplier is at the boundary of the cone $\tau = \pm |\xi|$, so the next stage is to dyadically decompose away from this region. We may consider the cases $\tau > 0$ and $\tau < 0$ separately, and via elementary changes of variables one can see that the latter case can be obtained from the former. Thus, we take $\psi \in C^\infty_c(\mathbb{R})$ to be supported in $[\frac{1}{2}, 2]$ such that
\begin{equation} \label{e:monodecomp}
s^{\alpha}=\sum_{k \in \mathbb{Z}} 2^{- k\alpha} \psi(2^{k} s)
\end{equation}
holds for all $s > 0$, and use this to decompose the multiplier as
\begin{align*}
\mathbf{1}_{\tau > 0} \phi(|\xi|)(|\xi| - \tau)_+^\alpha = m_0(\xi, \tau) + \sum_{k=k_0}^\infty 2^{- k\alpha} \mathbf{1}_{\tau > 0}\phi(|\xi|)\psi(2^{k} (|\xi|-\tau)).
\end{align*}
Here, of course, $m_0$ contains the terms up to $k_0 - 1$ of which only $O(1)$ remain thanks to the localisation in $\xi$, and thus $m_0$ is a smooth function supported in the set
\begin{equation*}\label{e:supp-m0}
\{ (\xi,\tau) \in \mathbb{R}^d \times \mathbb{R} : |\xi|\in [\tfrac{1}{2}, 2],\, |\xi|-\tau  \ge  2^{-k_0} \}.
\end{equation*}
The precise value of $k_0 \sim 1$ is not important and it will be clear that a sufficiently large choice can be made to make the following argument work. Associated with the above decomposition, we introduce the multiplier operator $\mathcal{C}_k$ given by
\[
\mathcal{F}(\mathcal C_k g)(\xi,\tau) = \mathbf{1}_{\tau > 0}\phi(|\xi|) \psi(2^{k} (|\xi|-\tau)) \widehat{g}(\xi,\tau).
\]
Since we are assuming $\alpha > \alpha^*$, we are reduced to proving
\begin{equation} \label{e:m0g}
\|  \mathcal F^{-1}( m_0   \widehat{g})\|_{L^q_tL_x^r}\lesssim \|g\|_{L^2}
\end{equation}
and
\begin{equation}\label{e:cdeltag}
\| \mathcal C_k  g\|_{L^q_tL_x^r}\lesssim  2^{k \alpha^*}\|g\|_{L^2} \qquad (k \geq k_0).
\end{equation}
Estimate \eqref{e:m0g} is more easily established since $m_0 \widehat{g}$ is compactly supported in a region where $|\xi| - \tau \sim 1$.
\begin{proof}[Proof of \eqref{e:m0g}]
Taking a function $\chi \in C^\infty_c(\mathbb{R}^d \times \mathbb{R})$ such that $\chi(\xi,\tau) = 1$ for all $(\xi,\tau)$ in this support, we may use the fact that $q,r \in [2,\infty)$ and the Young convolution inequality on mixed-norm spaces to see that
\[
\|  \mathcal F^{-1}( m_0   \widehat{g})\|_{L^q_tL_x^r} = \| \mathcal{F}^{-1}\chi *  \mathcal F^{-1}( m_0   \widehat{g})\|_{L^q_tL_x^r} \lesssim \|\mathcal F^{-1}( m_0   \widehat{g})\|_{L^2}.
\]
(Such an estimate is often referred to as Bernstein's inequality.) By Plancherel's theorem and since $\|m_0\|_{L^\infty} \lesssim 1$ we obtain \eqref{e:m0g}.
\end{proof}

\begin{proof}[Proof of \eqref{e:cdeltag}]
For $d \geq 4$, we use (in almost one fell swoop) the classical Strichartz estimates for the wave equation for frequency localised initial data. If we write
\begin{equation} \label{e:halfwave}
U(t)h(x) =  \int_{\mathbb{R}^d} e^{i(x \cdot \xi + t|\xi|)} \widehat{h}(\xi) \, \mathrm{d}\xi
\end{equation}
for the half-wave propagator, then the reader may find a proof of the following estimates in \cite{KeelTao} along with a more comprehensive historical account.
\begin{proposition}[Strichartz estimates for the wave equation] \label{p:Str}
Suppose $q,r \in [2,\infty)$ and $\frac{1}{q} \leq \frac{d-1}{2}(\frac{1}{2} - \frac{1}{r})$. Then
\[
\| U(t)h \|_{L^q_tL^r_x} \lesssim \| h \|_{L^2}
\]
whenever $\widehat{h}$ is supported in $\mathfrak{A}_0$.
\end{proposition}
\begin{remark}
In fact, the single frequency Strichartz estimate in Proposition \ref{p:Str} holds if and only if $d \geq 2$, $q,r \in [2,\infty]$, $\frac{1}{q} \leq \frac{d-1}{2}(\frac{1}{2} - \frac{1}{r})$ and $(q,r,d) \neq (2,\infty,3)$; thus, the cases where $q=\infty$ or $r = \infty$ are valid, except for the special case $(q,r,d) = (2,\infty,3)$, and consequently we may extend Theorem \ref{t:coneestimates} to include such exponents.
\end{remark}

Suppose now $k \geq k_0$. By Plancherel's theorem, we obviously have $\| \mathcal{C}_k g \|_{L^2} \lesssim \|g\|_{L^2}$. Hence, to prove \eqref{e:cdeltag}, by interpolation it is sufficient to show
\begin{equation} \label{e:sqrtdeltaboundg}
\|\mathcal{C}_k g\|_\lqr \lesssim  2^{-k/2} \|g\|_{L^2}
\end{equation}
provided $\frac{1}{q} \le  \frac{d-1}{2}(\frac12-\frac1r)$. By a simple change of variables, $\mathcal{C}_k g$ can be written as
\[
\mathcal{C}_k g(x,t) = \frac{1}{(2\pi)^{d+1}}\int_\mathbb{R}  e^{-i st}\psi(2^k s) \int_{\mathbb{R}^d} e^{i(x\cdot\xi+t|\xi|)} \phi(|\xi|) \mathbf{1}_{|\xi| \geq s} \widehat{g}(\xi, |\xi|-s) \, \mathrm{d}\xi \mathrm{d}s
\]
so that by applying Minkowski's integral inequality and Proposition \ref{p:Str} we obtain
\begin{align*}
\|\mathcal{C}_k g\|_\lqr \lesssim \int_\mathbb{R}  |\psi(2^k s)| \bigg(\int_{\mathbb{R}^d} |\widehat{g}(\xi, |\xi|-s)|^2 \, \mathrm{d}\xi\bigg)^{1/2} \mathrm{d}s.
\end{align*}
Estimate \eqref{e:sqrtdeltaboundg} now readily follows from the Cauchy--Schwarz inequality, further elementary changes of variables and another use of Plancherel's theorem.

The above argument fails to give the full range of claimed estimates for $d=2,3$ and further justification in these cases is given below. When $d=2$ the endpoint Strichartz estimate occurs when $q = 4$ and interpolation with $(q,r) = (2,2)$ is not sufficient to obtain the full range. The problem when $d=3$ is the failure of the endpoint Strichartz estimate at $(q,r) = (2,\infty)$ and thus the above argument fails to generate the claimed estimates in \eqref{e:cdeltag} when $q=2$ and $r \in (2,\infty)$.

To complete the proof of \eqref{e:cdeltag} when $d=2$, it clearly suffices to establish the additional estimate
\begin{equation*}
\| \mathcal{C}_kg\|_{L^2_tL^\infty_x} \lesssim 2^{-k/4}\|g\|_{L^2}
\end{equation*}
or, by duality,
\begin{equation} \label{e:two}
\| \mathcal{C}_kg\|_{L^2} \lesssim 2^{-k/4}\|g\|_{L^2_tL^1_x}.
\end{equation}
If we let $K$ be given by $\widehat{K}(\xi,\tau) = \phi(|\xi|)^2 \psi(2^k(|\xi| - \tau))^2$, then
$
\| \mathcal{C}_kg\|_{L^2}^2 \leq \langle K * g,g \rangle
$
and it suffices to show
$$
\|K\|_{L^1_tL^\infty_x} \lesssim 2^{-k/2}.
$$
Since
\[
|K(x,t)| \sim 2^{-k} |\widehat{\psi^2}(2^{-k}t)| \bigg|\int_{\mathbb{R}^2} \phi(|\xi|)^2 e^{i(x \cdot \xi - t|\xi|)} \, \mathrm{d}\xi\bigg|
\]
it follows that
\[
|K(x,t)| \lesssim 2^{-k} |\widehat{\psi^2}(2^{-k}t)|  (1 + |t|)^{-1/2}
\]
uniformly in $x \in \mathbb{R}^2$. This follows directly from the well-known dispersive estimate which plays a key role in the standard proof of the estimates in Proposition \ref{p:Str} (see, for example, \cite{KeelTao}). Hence
\[
\|K\|_{L^1_tL^\infty_x}  \lesssim \int_\mathbb{R} |\widehat{\psi^2}(s)|  (1 + 2^k|s|)^{-1/2} \,\mathrm{d}s \lesssim 2^{-k/2}
\]
as desired.

Finally, we prove \eqref{e:cdeltag} when $d=3$. As mentioned already, we may follow the above argument for $d \geq 4$ and we obtain all desired estimates except when $q=2$ and $r \in (2,\infty)$; in other words, by duality, it remains to prove
\begin{equation} \label{e:3Dfix} 
\|\mathcal{C}_k g\|_{L^2} \lesssim 2^{k(\frac{1}{r} - \frac{1}{2})}\|g\|_{L^2_tL^{r'}_x}
\end{equation}
for $r \in (2,\infty)$. By Plancherel's theorem, elementary changes of variables and a smooth partition of unity, we may write
$
\|\mathcal{C}_k g\|_{L^2}^2 = 2^{-k} \sum_{\ell \in \mathbb{Z}} \mathcal{I}_\ell, 
$
where
\begin{equation*}
\mathcal{I}_\ell =  \int_{\mathbb{R}^{d}} \int_{\mathbb{R}^{2(d+1)}}\vartheta(\tfrac{t-s}{2^\ell}) \widehat{\psi^2} (\tfrac{t-s}{2^k})  \phi^2(|\xi|) g(x,t)\overline{g(y,s)} e^{i(x-y)\cdot \xi} e^{i(t-s)|\xi|} \, \mathrm{d}x\mathrm{d}t \mathrm{d}y\mathrm{d}s \mathrm{d}\xi 
\end{equation*}
and $\vartheta \in C^\infty_c(\mathbb{R})$ is supported in $[\frac{1}{2},2]$.

On the one hand, using the Cauchy--Schwarz inequality, Plancherel's theorem and the rapid decay of $\widehat{\psi^2}$, we obtain
\begin{equation*} 
|\mathcal{I}_\ell| \leq C_N 2^\ell \min\{1, 2^{(k-\ell)N}\}  \|g\|_{L^2_{t,x}}^2
\end{equation*}
for all $N \geq 1$, where $C_N$ is some finite constant. On the other hand, if instead we make use of the dispersive estimate for the half-wave propagator, then we obtain
\begin{equation*}
|\mathcal{I}_\ell| \lesssim \int_{\mathbb{R}^{2(d+1)}} \vartheta(\tfrac{t-s}{2^\ell}) |\widehat{\psi^2}| (\tfrac{t-s}{2^k})   |g(x,t)||g(y,s)| \frac{1}{|t-s|} \, \mathrm{d}x\mathrm{d}t \mathrm{d}y\mathrm{d}s
\end{equation*}
and hence
\begin{equation*} 
|\mathcal{I}_\ell| \leq C_N \min\{1, 2^{(k-\ell)N}\}      \|g\|_{L^2_tL^1_x}^2
\end{equation*}
for all $N \geq 1$. By interpolation we obtain
\[
|\mathcal{I}_\ell| \leq C_N \min\{1, 2^{(k-\ell)N}\} 2^{\frac{2\ell}{r}} \|g\|_{L^2_tL^{r'}_x}^2
\]
for all $r \in (2,\infty)$, and this estimate easily gives \eqref{e:3Dfix}.
\end{proof}

\begin{remark}  If  $\frac1q>\frac{d-1}2(\frac12-\frac1r)$, $2<q<\infty$,  
we can prove weak type estimates for $\mathcal{C}^\alpha$ at the critical exponent $\alpha=\alpha^*(q,r)$. In fact, for 
$q,r$ as above we have 
\begin{equation}\label{e:weaktype} 
\|\mathcal{C}^{\alpha^*}\!\!\!g\|_{L^{q,\infty}_tL^r_x} \lesssim \|g\|_{L^2} .
\end{equation}
In the pure-norm case ($q=r$), this can be strengthened to the strong type estimate 
$$
\|\mathcal{C}^{\alpha^*}\!\!\!g\|_{L^q_{x,t}} \lesssim \|g\|_{L^2}
$$
and we refer the reader to \cite{Lee} for details of how this upgrade proceeds. (It seems likely that the same also holds for the mixed-norm estimate but we do not pursue this here.) Since $L^{r/2,\infty}_t$ is normable, \eqref{e:weaktype} combined with the Littlewood--Paley inequality gives 
\[ 
\| \diffp\diffn  \rho f\|_{L^{q,\infty}_tL_x^r}\lesssim \|f\|_{L^2}
\] 
with $\beta_-=\beta_-^*$ provided that $\frac1q>\frac{d-1}2(\frac12-\frac1r)$, $2<q<\infty$. 

The proof  of \eqref{e:weaktype} is rather elementary (and follows from a more general principle which may be found, for example, in \cite{LeeSeo}).  Indeed, we may assume $\|g\|_{L^2}=1$ and it suffices to show  
\begin{equation}\label{e:weak} 
\bigg|\bigg\{ t: \bigg\|\sum_{k=k_0}^\infty 2^{-\alpha^*\! k}  \mathcal C_k g\bigg\|_{L^r_x}\ge \lambda\bigg\}\bigg| \lesssim \lambda^{-q}.
 \end{equation}
Choose $q_1, q_2\in (2,\infty)$ such that  $q_1<q<q_2$. So, we have $  -\alpha^*+\alpha^*(q_1,r)>0>-\alpha^*+\alpha^*(q_2,r)$. 
Hence,  Minkowski's  inequality followed by \eqref{e:cdeltag} yields 
\[ 
\bigg\|\sum_{k=k_0}^{N-1} 2^{-\alpha^*\! k}  \mathcal C_k g\bigg\|_{L^{q_1}_tL^r_x}\lesssim   2^{N(\alpha^*(q_1,r)-\alpha^*)}
\]
and
\[
\bigg \|\sum_{k=N}^\infty 2^{-\alpha^*\! k}  \mathcal C_k g\bigg\|_{L^{q_2}_tL^r_x}\lesssim   2^{N(\alpha^*(q_2,r)-\alpha^*)} .
 \] 
So, by this and Chebyshev's inequality, the left-hand side of \eqref{e:weak} is bounded by 
\begin{align*}
 \bigg|\bigg\{ t: \bigg\|\sum_{k=k_0}^{N-1}  2^{-\alpha^*\! k}  \mathcal C_k g\bigg\|_{L^r_x} & \ge \frac\lambda 2\bigg\}\bigg| + \bigg|\bigg\{ t:\, \bigg\|\sum_{k=N}^\infty 2^{-\alpha^*\! k}  \mathcal C_k g\bigg\|_{L^r_x}\ge \frac\lambda 2\bigg\}\bigg| \\
& \lesssim 2^{q_1N(\alpha^*(q_1,r)-\alpha^*)} \lambda^{-q_1}  +2^{q_2N(\alpha^*(q_2,r)-\alpha^*)}\lambda^{-q_2} . 
\end{align*} 
Choosing $N$ which optimises the last expression gives \eqref{e:weak}.
\end{remark}

\subsection{Necessity}
By duality, it suffices to show that
\begin{equation} \label{e:spherenec1}
\alpha \ge \frac{1}{q}+\frac{d-1}{2r}-\frac{d+1}{4}
\end{equation}
and
\begin{equation} \label{e:spherenec2}
\alpha > -\frac{1}{2}
\end{equation}
are necessary conditions for
\begin{equation} \label{e:conedual}
\| \mathcal{C}^\alpha g \|_{L^2} \lesssim \| g \|_{L^{q'}_tL^{r'}_x}.
\end{equation}
We will accomplish these claims using a Knapp-type example and a bump function example as follows.

\subsection*{Proof of \eqref{e:spherenec1}}
Let $0<\delta\ll 1$ and $g_\delta$ be given by
\[
\widehat{g_\delta}(\xi,\tau) = \phi\bigg(\frac{\xi_d-\tau}\delta\bigg)\phi (\xi_d+\tau) \prod_{j=1}^{d-1} \phi\bigg(\frac{\xi_j}{\sqrt\delta}\bigg).
\]
Note that for $(\xi,\tau)$ in the support of $g_\delta$ one has
\[
|\xi_j| \sim \sqrt{\delta} \,\,\,\, (1 \leq j \leq d-1), \quad \tau, |\xi| \sim 1, \quad  |\xi| - \tau \sim \delta, \quad |\xi' - e_d| \sim \sqrt{\delta}
\]
and thus $\widehat{g_\delta}$ is a smooth function adapted to a $\delta$-plate. Writing $\theta_\delta$ for the support of $\widehat{g_\delta}$, by Plancherel's theorem we clearly have
$
\| \mathcal{C}^\alpha g_\delta \|_{L^2} \sim \delta^{\alpha} |\theta_\delta|^{1/2}.
$
Also, one can show that the main contribution to the right-hand side of \eqref{e:conedual} arises from the dual box consisting of those $(x,t)$ such that
\[
|x_1|, \dots, |x_{d-1}| \lesssim   \frac{1}{\sqrt{\delta}},\quad |x_d+t| \lesssim 1, \quad |x_d-t| \lesssim \frac{1}{\delta}
\]
and therefore
$
\| g \|_{L^{q'}_tL^{r'}_x} \sim |\theta_\delta| \delta^{-\frac{d-1}{2r'} - \frac{1}{q'}}.
$
Since $|\theta_\delta| \sim \delta^{\frac{d+1}{2}}$, it follows that if \eqref{e:conedual} holds then $\alpha \geq \frac{d+1}{4} - \frac{d-1}{2r'} - \frac{1}{q'}$, which is equivalent to \eqref{e:spherenec1}.

\subsection*{Proof of \eqref{e:spherenec2}}
Choose $g(x,t) = g_1(x)g_2(t)$, where $g_1 \in C^\infty_c(\mathbb{R}^d)$ is such that $\widehat{g}_1(\xi) = 1$ for $|\xi| \in [\frac{1}{2},2]$ and $g_2 \in C^\infty_c(\mathbb{R})$ is such that $\widehat{g}_2(\tau) = 1$ for all $|\tau| \leq 2$. By Plancherel's theorem and a trivial change of variables in $\tau$,
\begin{align*}
\| \mathcal{C}^\alpha g \|_{L^2}^2 \sim \int_{|\tau| \leq |\xi|} \phi(|\xi|)^2 (|\xi| - |\tau|)^{2\alpha} |\widehat{g}_1(\xi)|^2 |\widehat{g}_2(\tau)|^2 \, \mathrm{d}\tau \mathrm{d}\xi \gtrsim \int_0^1 (1 - \lambda)^{2\alpha} \, \mathrm{d}\lambda
\end{align*}
and hence $\alpha > -\frac{1}{2}$.

\section{Approach II : Direct analysis} \label{section:direct}
In this section, we shall focus on the case $V = \mathbb{S}^{d-1}$; later we make some remarks on the robustness of the approach taken here and applicability to other velocity domains.

Consider initial data $f$ belonging to the (homogeneous) Besov space $\dot{B}_{p,2}^s$. To define this space, we use the Littlewood--Paley projection operators $(P_j)_{j \in \mathbb{Z}}$ introduced in Section \ref{section:pre}, and the norm
\[
\|f\|_{\dot{B}_{p,2}^s} = \bigg( \sum_{j \in \mathbb{Z}} 2^{2js} \|P_jf\|_{L^p}^2 \bigg)^{1/2}.
\]
For instance, $\dot{B}_{2,2}^s$ is the (homogeneous) fractional Sobolev space $\dot{H}^s$ (regularity measured in the spatial variable) and specialising further still $\dot{B}_{2,2}^0$ is $L^2$.

We begin by considering $f$ such that $\widehat{f}(\cdot,v)$ is supported in $\mathfrak{A}_0$ for each $v \in \mathbb{S}^{d-1}$. Throughout this section, we regard $p,q,r$ and $s$ as given parameters, and we set
\begin{equation} \label{e:sscaling}
\beta_+ + \beta_- =  s + \frac{d}{r} +\frac{1}{q} - \frac{d}{p}.
\end{equation}
The core argument in this section is based to some extent on the above proof of Theorem \ref{t:coneestimates}; the key role played by the Strichartz estimates for the wave equation above will be replaced by different estimates (such as the sharp $\ell^p$ decoupling inequality) depending on the context.
%Also, we fix $\beta_-$ such that $\beta_- > \beta^*_-$.

We shall often use that $\widehat{\rho f}$ is supported in the region $\mathfrak{C}$. Thus, as a result of the spatial localisation of the initial data we shall see that the key estimates are
\begin{equation} \label{e:m0}
\|  \mathcal F^{-1}( m_0   \widehat{\rho f})\|_{L^q_tL_x^r}\lesssim \|f\|_{L^p}
\end{equation}
and
\begin{equation}\label{e:cdelta}
\| \mathcal C_k  \rho f\|_{L^q_tL_x^r}\lesssim  2^{k \eta}\|f\|_{L^p} \qquad (k \geq k_0).
\end{equation}
Here $\eta \sim 1$ is a crucial parameter which determines the range of admissible $\beta_-$, the multiplier $m_0$ is given by
\[
m_0(\xi,\tau) = \sum_{k \leq k_0 - 1} 2^{-k\beta_-} \phi(|\xi|)\psi(2^k(|\xi| - \tau))
\]
and $k_0 \sim 1$ is chosen sufficiently large. We emphasise that there are $O(1)$ terms in the sum defining $m_0$ thanks to the localisation to $\mathfrak{A}_0$.

As one may expect, the estimate \eqref{e:m0} away from the singularity in the multiplier is more easily established.
\begin{lemma} \label{l:easybit}
If $p \in [2,\infty]$ and $q,r \in [p,\infty]$, then \eqref{e:m0} holds.
\end{lemma}
\begin{proof}
We use interpolation between the cases $p=2$ and $p=\infty$. For $p=2$, since $|\xi|-\tau \sim 1$ on the support of $m_0$, by \eqref{e:nodeltarep} and the Cauchy--Schwarz inequality, we have
 \begin{align*}
|m_0(\xi,\tau)\widehat{\rho f}(\xi,\tau)|^2
& \lesssim
\bigg| \int_{\Sigma_{\xi,\tau} }    \widehat{f}(\xi,v) \, \mathrm{d}\sigma_{\xi,\tau}(v)  \bigg|^2
 \\
& \lesssim
\int_{\Sigma_{\xi,\tau} }    |\widehat{f}(\xi,v)|^2 \, \mathrm{d}\sigma_{\xi,\tau}(v)
 \\
& \sim \int_{\mathbb S^{d-1}}  \delta(\tau+v\cdot\xi)    |\widehat{f}(\xi,v)|^2 \, \mathrm{d}\sigma(v).
\end{align*}
Hence integration in $\tau$ and then $\xi$ gives
$
\|m_0   \widehat{\rho f}\|_{L^2} \lesssim \|f\|_{L^2}
$
and therefore \eqref{e:m0} follows for $p = q = r =2$.

For $p=\infty$, since $\mathcal{F}^{-1}m_0 \in L^1$ and, trivially, $\|\rho f\|_{L^\infty} \lesssim \|f\|_{L^\infty}$ we obtain \eqref{e:m0} when $p = q = r = \infty$. Interpolating between these two estimates we obtain that \eqref{e:m0} is true whenever $p = q = r \in [2,\infty]$, and since $m_0$ is a bounded function of compact support we finally obtain
\begin{equation*}
\| \mathcal F^{-1}( m_0   \widehat{\rho f})\|_{L^q_tL_x^r} \lesssim \| \mathcal F^{-1}( m_0   \widehat{\rho f})\|_{L^p} \lesssim \|f\|_{L^p}
\end{equation*}
as desired.
\end{proof}
The following conditional result clarifies the decisive nature of the estimates \eqref{e:m0} and \eqref{e:cdelta}.
\begin{proposition} \label{p:abstract}
Suppose \eqref{e:m0} and \eqref{e:cdelta} hold. Then whenever $q,r \in [2,\infty)$ and $\beta_- > \eta$ the estimate
\[
\|D_+^{\beta_+}D_-^{\beta_-} \rho f\|_{L^q_tL^r_x} \lesssim \|f\|_{\dot{B}^s_{p,2}}
\]
holds for all $f \in \dot{B}^s_{p,2}$.
\end{proposition}
\begin{proof}
As in the proof of Theorem \ref{t:coneestimates}, using the identity \eqref{e:monodecomp}, the triangle inequality, and the estimates \eqref{e:m0} and \eqref{e:cdelta}, we immediately obtain
\begin{equation*}
\|D_-^{\beta_-} \rho (P_0f)\|_{L^q_tL^r_x} \lesssim \|P_0f\|_{L^p}
\end{equation*}
since $\beta_- > \eta$. It follows from the frequency support of $\rho(P_0f)$ that
\begin{equation*}
\|D_+^{\beta_+}D_-^{\beta_-} \rho (P_0f)\|_{L^q_tL^r_x} \lesssim \|P_0f\|_{L^p}
\end{equation*}
and then a rescaling argument shows that
\begin{equation} \label{e:Besovkpart}
\|D_+^{\beta_+}D_-^{\beta_-} \rho (P_jf)\|_{L^q_tL^r_x} \lesssim 2^{js} \|P_jf\|_{L^p}.
\end{equation}
The basis of this rescaling argument are the identities
\[
D_+^{\beta_+}D_-^{\beta_-} \rho (P_jf)(x,t) = 2^{j(\beta_+ + \beta_- + d)} D_+^{\beta_+}D_-^{\beta_-} \rho (P_0f_j)(2^jx,2^jt)
\]
and
$
P_0f_j(x,v) = 2^{-jd}P_jf(2^{-j}x,v),
$
where $\widehat{f_j}(\xi,v) = \widehat{f}(2^j\xi,v)$; these are easily verified by simple changes of variables.

Since $q,r \in [2,\infty)$, it follows from \eqref{e:LP} that
\begin{equation} \label{e:decoupling}
\| \diffp\diffn  \rho f \|_{L^q_tL_x^r} \leq \bigg( \sum_{j \in \mathbb{Z}} \| \diffp\diffn  \rho (P_jf) \|_{L^q_tL_x^r}^2 \bigg)^{1/2}
\end{equation}
and hence the desired estimate $\| \diffp\diffn  \rho f \|_{L^q_tL_x^r} \lesssim \|f\|_{\dot{B}^s_{p,2}}$ follows directly from \eqref{e:Besovkpart}.
\end{proof}

The above argument focuses attention onto the estimate \eqref{e:cdelta}. In this section and the subsequent section we shall exhibit a variety of smoothing estimates based on Proposition \ref{p:abstract}, in each case our work has been reduced to verifying \eqref{e:cdelta}. We begin with a proof of (the sufficiency claims in) Theorem \ref{t:thm-sobs}.

\subsection{Direct proof of Theorem \ref{t:thm-sobs}} \label{subsection:directproof}
Suppose $d \geq 2$, $q,r\in [2,\infty)$ and assume $\beta_+$, $\beta_-$ satisfy \eqref{e:scaling} with $\beta_- > \beta_-^*$. Since $p=2$ and $s=0$, thanks to Lemma \ref{l:easybit} and Proposition \ref{p:abstract}, it suffices to prove \eqref{e:cdelta} with $\eta = \beta_-^*$. By \eqref{e:cdeltag}, it suffices to prove
\begin{equation} \label{e:cdelta22}
\| \mathcal C_k  \rho f\|_{L^2_tL_x^2}\lesssim  2^{\frac{3-d}{4} k}\|f\|_{L^2} \qquad (k \geq k_0).
\end{equation}
To see this, note that the representation in \eqref{e:nodeltarep} allows us to write
\begin{equation*}
|\widehat{\rho f}(\xi,\tau)| \sim 2^k \bigg|\int_{\Sigma_{\xi,\tau}} \widehat{f}(\xi,v) \, \mathrm{d}\sigma_{\xi,\tau}(v)\bigg|
\end{equation*}
whenever $|\xi| - \tau \sim 2^{-k}$. Since we also assume $|\xi| \sim 1$, we have that $\Sigma_{\xi,\tau}$ is a $(d-2)$-dimensional sphere with radius $(1 - \frac{\tau^2}{|\xi|^2})^{1/2} \sim 2^{-k/2}$, and hence the Cauchy--Schwarz inequality yields
\begin{align*}
|\widehat{\rho f}(\xi,\tau)|^2 \lesssim 2^{-k \frac{d-3}{2}} \int_{\mathbb{S}^{d-1}} \delta(v \cdot \xi + \tau) |\widehat{f}(\xi,v)|^2 \, \mathrm{d}\sigma(v).
\end{align*}
Hence, integrating in $\tau$, then $\xi$, we get \eqref{e:cdelta22}. This completes our direct proof of Theorem \ref{t:thm-sobs}.

\begin{remark}
For simplicity of exposition, we have presented the direct approach with the velocity domain as $\mathbb{S}^{d-1}$ equipped with Lebesgue measure. However, it is clear that the approach is sufficiently robust to handle other situations. For example, we may follow the above proof to give an alternative proof of the more general statement in Theorem \ref{t:thm-sobsgen} concerned with the family of measures $\mathrm{d}\mu_\kappa(v) = w_\kappa(v) \, \mathrm{d}v$, where $w_\kappa$ is given by \eqref{e:wkappa}.
\end{remark}

\subsection{Besov space estimates via the $\ell^p$ decoupling inequality}
Here we show how the recently established sharp decoupling theorems of Bourgain and Demeter induce $\dot{B}^s_{p,2} \to L^q$ smoothing estimates for $\rho$. Since the mixed-norm theory of decoupling estimates has currently not been fully developed, we shall consider only the pure-norm where $q=r$ on the velocity average; it will be obvious how to extend our results to the mixed-norm case on the basis of a mixed-norm extension of Theorem \ref{t:decoupling} below (in particular, see Lemma \ref{l:decouplingimplies}).

In order to state the $\ell^p$ decoupling inequality, it is necessary to introduce some notation, starting with
\[
\Gamma = \{ (\xi,|\xi|) \in \mathbb{R}^d \times \mathbb{R} : |\xi| \in [1,2]\}
\]
for the truncated cone and
\[
\mathcal{N}_k(\Gamma) = \{ (\xi,\tau) \in \mathbb{R}^d \times \mathbb{R} : \tau \in [1,2], |\tau - |\xi|| \leq 2^{-k} \}
\]
for the $2^{-k}$ neighbourhood of $\Gamma$, with $k \gg 1$. Then, subordinate to a given $2^{-k/2}$-separated family of points on the sphere $\mathbb{S}^d$, we let $\mathcal{P}_k(\Gamma)$ be the partition of $\mathcal{N}_k(\Gamma)$ into plates $\theta$ with height $O(1)$, thickness $O(2^{-k})$ in the normal direction, and $O(2^{-k/2})$ in the remaining $d-1$ directions.

To define an important exponent $\gamma(p,q)$ in the following decoupling theorem, we introduce the notation $\mathcal{T} = \mathcal{T}^0 \cup \mathcal{T}_0$, where $\mathcal{T} = \{(\frac{1}{p},\frac{1}{q}) \in [0,\frac{1}{2}]^2: \frac{1}{p} \geq \frac{1}{q}\}$, $\mathcal{T}^0 = \{ (\frac{1}{p},\frac{1}{q}) \in \mathcal{T} : \frac{1}{q} \geq \frac{d-1}{2(d+1)} \}$ and $\mathcal{T}_0 = \mathcal{T} \setminus \mathcal{T}^0$. Then $\gamma(p,q)$ is set by
\[
\gamma(p,q) = \left\{ \begin{array}{llllll} \frac{d+1}{2q} + \frac{d-1}{4}- \frac{d}{p} & \qquad \text{if $(\frac{1}{p},\frac{1}{q}) \in \mathcal{T}^0$} \vspace{2mm} \\ \frac{d-1}{2} - \frac{d}{p} & \qquad \text{if $(\frac{1}{p},\frac{1}{q}) \in \mathcal{T}_0$.} \end{array} \right.
\]
Also, for each $\theta \in \mathcal{P}_k(\Gamma)$ and $k \gg 1$, we define the projection $\Pi_{k,\theta}$ by
\[
\mathcal{F}(\Pi_{k,\theta}g)(\xi,\tau) = \chi_\theta(\xi') \psi(2^k(|\xi| - \tau)) \widehat{g}(\xi,\tau),
\]
where $\chi_\theta$ is a smooth cut-off function supported on the corresponding subset of $\mathbb{S}^{d}$.
\begin{theorem} \label{t:decoupling}
Suppose that $(\frac{1}{p},\frac{1}{q}) \in \mathcal{T}$. Then for each $\varepsilon > 0$ there exists a constant $C_\varepsilon < \infty$ such that
\begin{equation} \label{e:decoupling}
\|g\|_{L^q} \leq C_\varepsilon 2^{(\gamma(p,q) + \varepsilon)k} \bigg( \sum_{\theta \in \mathcal{P}_k(\Gamma)} \|\Pi_{k,\theta}g\|_{L^p}^p \bigg)^{1/p}
\end{equation}
whenever $\widehat{g}$ is supported in $\mathcal{N}_k(\Gamma)$.
\end{theorem}
In the diagonal case $p=q$, Theorem \ref{t:decoupling} is due to Bourgain and Demeter \cite{BD} (these estimates are also known in the literature as \textit{Wolff's inequalities}, and earlier contributions were made in \cite{GarrigosSeeger}, \cite{LabaWolff}, \cite{Wolff}). For example, this may be obtained from Theorem 1.2 in \cite{BD} where a stronger statement is proved with an $\ell^2$ norm on the right-hand side; the $\ell^p$ decoupling estimate \eqref{e:decoupling} follows immediately by H\"older's inequality. Note also that when $p=2$, as a consequence of \eqref{e:cdeltag}, we may obtain
\begin{equation*}
\|g\|_{L^q} \lesssim 2^{\alpha^*(q,q)k} \|g\|_{L^2}
\end{equation*}
and it is easily checked that $\gamma(2,q) = \alpha^*(q,q)$. Hence, the full range of estimates \eqref{e:decoupling} for $(\frac{1}{p},\frac{1}{q})$ in the triangle $\mathcal{T}$ follows by interpolating between the hypotenuse $\frac{1}{p}=\frac{1}{q}$ and the vertical edge $\frac{1}{p}=\frac{1}{2}$. Of course, this argument shows that there is no loss of arbitrary $\varepsilon > 0$ in the exponent in \eqref{e:decoupling} when $p=2$; however, the loss for general $p$ and $q$ is completely inconsequential in our application of Theorem \ref{t:decoupling} below, since the exponent from \eqref{e:decoupling} will manifest itself in the exponent in the induced estimate \eqref{e:cdelta} and the subsequent summation of a geometric series already necessitates an open range for $\beta_-$.

\begin{lemma} \label{l:decouplingimplies}
Suppose $\varepsilon > 0$. Then estimate \eqref{e:decoupling} implies \eqref{e:cdelta}, with a bound depending on $\varepsilon$, when $q=r$ and $\eta = \gamma(p,q) + \frac{3-d}{2p} + \varepsilon$.
\end{lemma}
\begin{proof}
In order to prove \eqref{e:cdelta}, we assume that $\widehat{f}(\cdot,v)$ is supported in $\mathfrak{A}_0$ for each $v \in \mathbb{S}^{d-1}$. First, we claim
\begin{equation} \label{e:eachtheta}
\| \Pi_{k,\theta} \rho f \|_{L^p} \lesssim 2^{\frac{3-d}{2p}k}\|\widetilde{\Pi}_\theta f\|_{L^p}
\end{equation}
for each $\theta \in \mathcal{P}_k(\Gamma)$. Here we write $\widetilde{\Pi}_\theta$ for the operator given by
\[
\mathcal{F}(\widetilde{\Pi}_\theta f)(\xi,v) = \widetilde{\chi}_\theta(\xi') \widehat{f}(\xi,v)
\]
where $\widetilde{\chi}_\theta$ is a smooth cut-off function such that $\chi_\theta = \chi_\theta\widetilde{\chi}_\theta$. Since
$$
\widetilde{\chi}_\theta(\xi') \widehat{\rho f}(\xi,\tau) = \mathcal{F}(\rho(\widetilde{\Pi}_\theta f))(\xi,\tau)
$$
we have $\Pi_{k,\theta} \rho f = \Pi_{k,\theta} \rho(\widetilde{\Pi}_{\theta} f)$, and we may directly apply \eqref{e:cdelta22} to show that \eqref{e:eachtheta} holds when $p=2$. Since $\rho$ is trivially a bounded operator $L^\infty \to L^\infty$, estimate \eqref{e:eachtheta} also holds when $p = \infty$, and the claim follows.

Applying \eqref{e:decoupling} and subsequently using \eqref{e:eachtheta}, it follows that for any $\varepsilon > 0$ we have the estimate
\begin{equation*}
\| \mathcal C_k  \rho f\|_{L^q} \leq C_\varepsilon 2^{(\gamma(p,q) + \frac{3-d}{2p} + \varepsilon)k}  \bigg( \sum_{\theta \in \mathcal{P}_k(\Gamma)} \|\widetilde{\Pi}_{\theta} f\|_{L^p}^p \bigg)^{1/p}
\end{equation*}
for some constant $C_\varepsilon < \infty$. By again considering $p=2$ and $p=\infty$, and once again making use of the fact that $\widehat{f}$ has support in $\mathfrak{A}_0$, one can show that $\sum_{\theta} \|\widetilde{\Pi}_{\theta} f\|_{p}^p \lesssim \|f\|_{p}^p$, thus completing our proof of the lemma.
\end{proof}
Proposition \ref{p:abstract} now immediately yields the following.
\begin{theorem} \label{t:Besov}
Suppose $d \geq 2$, $p \in [2,\infty)$, $q \in [p,\infty)$ and $s \in \mathbb{R}$. If $\beta_+, \beta_-$ satisfy \eqref{e:sscaling} and $\beta_- > \gamma(p,q) + \frac{3-d}{2p}$, then
\[
\|D_+^{\beta_+}D_-^{\beta_-} \rho f\|_{L^q} \lesssim \|f\|_{\dot{B}^s_{p,2}}
\]
holds for all $f \in \dot{B}^s_{p,2}$.
\end{theorem}

\section{Further results and remarks} \label{section:furthers}
\subsection{Multilinear velocity averaging}
The decoupling inequalities developed recently by Bourgain, Demeter and Guth, such as Theorem \ref{t:decoupling} above (see also \cite{BDG}), draw on recent developments in multilinear harmonic analysis, and in particular the fact that certain \emph{multilinear} estimates for the cone multiplier $\mathcal{C}_k$ are available in essentially optimal form. 
Such multilinear inequalities rely crucially on the multilinear Kakeya-type inequalities established in \cite{BCT}; see also \cite{CV}, \cite{Guthendpoint}, \cite{Guthshort}, \cite{Zhang}, \cite{BBFL}. Kakeya-type inequalities, being $X$-ray transform estimates, are themselves naturally formulated in terms of the kinetic transport equation and the velocity-averaging operator $\rho$; recall that the dual operator $\rho^*$ given by \eqref{rhod} %(THIS IS (3.2))
is simply a space-time $X$-ray transform. This perspective is somewhat implicit in the literature; see for example \cite{LT} and \cite{WolffX}. 
In multilinear settings, Kakeya-type inequalities are much better understood than their classical linear counterparts, and in some instances may be expressed quite directly as Strichartz estimates for the kinetic transport equation. Most notably, an elementary limiting argument reveals that the \emph{affine-invariant endpoint} multilinear Kakeya inequality (see \cite{BG} and \cite{CV}) 
is equivalent to the null-form estimate
\begin{equation}\label{general form kinetic}
\int_{\mathbb{R}}\int_{\mathbb{R}^d}\widetilde{\rho}(f_1,\hdots,f_{d+1})(t,x)^{1/d}\,\mathrm{d}x\mathrm{d}t\lesssim\prod_{j=1}^{d+1}\|f_j\|_{L^1_{x,v}}^{1/d}
\end{equation}
where
$$
\widetilde{\rho}(f_1,\hdots,f_{d+1})(t,x) = \int_{(\mathbb{R}^d)^{d+1}}\prod_{j=1}^{d+1}f_j(x-tv_j,v_j) \mathfrak{V}(v_1,\ldots,v_{d+1}) \prod_{\ell=1}^{d+1} \mathrm{d}\mu_\ell(v_\ell).
$$
and
$$
\mathfrak{V}(v_1,\ldots,v_{d+1}) = \left|\det\left(
\begin{array}{ccc}
1& \cdots & 1\\
v_1 &
\cdots & v_{d+1}\\
\end{array}\right)\right|
$$
In the above, $\mu_1,\hdots, \mu_{d+1}$ denote compactly supported positive Borel measures on $\mathbb{R}^d$, and $\|f_j\|_{L^1_{x,v}}$ is given with respect to Lebesgue measure in the spatial variable and $\mu_j$ in the velocity variable. Also, we clarify that $\mathfrak{V}(v_1,\ldots,v_{d+1})$ coincides with the volume of the simplex in $\mathbb{R}^d$ with vertices $v_j \in \mathbb{R}^d$, $1 \leq j \leq d+1$. Here we interpret $\widetilde{\rho}(f_1,\hdots,f_{d+1})$ as a $(d+1)$-linear variant of the linear operator $\rho$ defined in \eqref{rhobasic}; notice that without the determinant factor the left-hand side of \eqref{general form kinetic} simply becomes
$$
\int_{\mathbb{R}}\int_{\mathbb{R}^d}\prod_{j=1}^{d+1}\rho(f_j)(t,x)^{1/d}\,\mathrm{d}x\mathrm{d}t
$$
and so \eqref{general form kinetic} represents $L^1$ control of a part of this expression. The inequality \eqref{general form kinetic} may be viewed as a generalisation (or perturbation) of the classical affine-invariant Loomis--Whitney inequality, as the special case of measures $\mu_1,\hdots,\mu_{d+1}$ supported at non-cohyperplanar \emph{points} in $\mathbb{R}^{d}$ quickly reveals.

There are other, more elementary, velocity-averaging inequalities which draw on this multilinear perspective. For example, if $\rho$ is given by \eqref{rhobasic}, we have
\[
\|\rho f\|_{L^{d+1}_{t,x}}^{d+1} = \int_{(\mathbb{R}^d)^{d+1}} \int_{\mathbb{R}^d}\int_{\mathbb{R}} \prod_{j=1}^{d+1} f(x-tv_j,v_j) \, \mathrm{d}t\mathrm{d}x \prod_{\ell=1}^{d+1} \mathrm{d}\mu(v_\ell)
\]
and an application of the classical affine-invariant Loomis--Whitney inequality (see, for example, \cite{BB}) in the variable $(x,t)$ reveals the bound
$$
\|\rho f\|_{L^{d+1}_{t,x}}\lesssim I_{1/d}(\mu)^{\frac{1}{d+1}}\|f\|_{L^\infty_vL^d_x}
$$
where
$$
I_{1/d}(\mu):=\int_{(\mathbb{R}^d)^{d+1}}\mathfrak{V}(v_1,\ldots,v_{d+1})^{-1/d}
\,\prod_{\ell=1}^{d+1} \mathrm{d}\mu(v_\ell).
$$
Such ``energy functionals" are related to the notion of affine dimension, and present a more geometric and measure theoretic perspective on velocity averaging. Multilinear determinant functionals of this type are studied in \cite{Drury}, \cite{Gressman} and \cite{Vald}.

\subsection{Symmetric data} Here we exhibit various ways in which the smoothness regime in the central estimates in this work, namely those in Theorem \ref{t:thm-sobs}, may be broadened if we impose some symmetry hypotheses on the initial data. Such a phenomenon is well-known in surrounding contexts, including the following Strichartz estimates for the wave equation for radially symmetric data, whose range of validity should be compared with Proposition \ref{p:Str}.
\begin{proposition} \label{p:radialStr}
Suppose $q,r \in [2,\infty)$ and $\frac{1}{q} < (d-1)(\frac{1}{2}-\frac{1}{r})$. Then
\[
\| U(t)h \|_{L^q_tL^r_x} \lesssim \| h \|_{L^2}
\]
whenever $\widehat{h}$ is radially symmetric and supported in $\mathfrak{A}_0$. Here $U(t)$ denotes the half-wave propagator given by \eqref{e:halfwave}.
\end{proposition}
We refer the reader to \cite{ChoOzawa}, \cite{FangWang2008}, \cite{HidanoKurokawa}, \cite{KM}, \cite{Sterbenz} for details. Below we establish some improved smoothing estimates for the velocity average $\rho$ acting on $L^2$ initial data which are radial in the spatial variable, and specialising further to initial data which are radial in the spatial variable and independent of the velocity variable; we denote these classes as $L^2_{\textrm{rad}(x)}$ and $L^2_{\textrm{rad}(x,v)}$. Again, we focus on the case $V = \mathbb{S}^{d-1}$.

These results will improve upon Theorem \ref{t:thm-sobs} for such classes of data and our approach will follow the direct analysis in Section \ref{section:direct}; we re-emphasise that, as shown in Subsection \ref{subsection:directproof}, the role of the Strichartz estimates for the wave equation is to establish \eqref{e:cdeltag} (applied to $g = \rho f$) which in turn allow us to work on $L^2$. For $f \in L^2_{\textrm{rad}(x)}$, it is not necessarily true that $\rho f$ is radially symmetric, thus the additional gain only arises in an improvement in \eqref{e:cdelta22}. However, if $f \in L^2_{\textrm{rad}(x,v)}$, then $\rho f$ is radially symmetric and yet further gain is available in \eqref{e:cdeltag} by exploiting Proposition \ref{p:radialStr}.

As outlined above, the direct approach rests on sharp estimates in the case $(q,r) = (2,2)$, and so we begin here. Our argument naturally leads to estimates beyond initial data in $L^2$ by introducing Sobolev regularity with respect to the velocity variable.
\begin{theorem} \label{t:roughdata}
Suppose $d \geq 2$, $s \in [-\frac{d-2}{2},0]$, $\beta_+ + \beta_- = \frac{1}{2}$ and
$
\beta_- > - s - \frac{d-2}{2}.
$
Then
\begin{equation*}
\| \diffp\diffn  \rho f\|_{L^{2}}\lesssim \| (1 - \Delta)^{s/2} f\|_{L^{2}}
\end{equation*}
holds for all $f \in L^2_{\emph{\textrm{rad}}(x)}$.
\end{theorem}
Here, $\Delta$ is the Laplace--Beltrami operator on $\mathbb{S}^{d-1}$ acting on the velocity variable. As a simple comparison, taking $s=0$, we see that the range $\beta_- > \frac{2-d}{2}$ is allowed for \eqref{e:sobs} for $f \in L^2_{\textrm{rad}(x)}$, extending the range $\beta_- \geq \frac{3-d}{4}$ for general $f \in L^2$.

Now fix $f \in L^2_{\textrm{rad}(x)}$ and write
$
\widehat{f}(\xi,v) = F_0(|\xi|,v)
$
for $f \in L^2_{\textrm{rad}(x)}$. Then, for each $r > 0$, we have the representation
\begin{equation} \label{e:radialexpansion}
F_0(r,v) = \sum_{k=0}^\infty Y_k^{r}(v)
\end{equation}
in terms of the basis of spherical harmonics for $L^2(\mathbb{S}^{d-1})$. Using polar coordinates, we may then write
\[
\| (1 - \Delta)^{s/2} f\|_{L^{2}}^2 = \frac{|\mathbb{S}^{d-1}|}{(2\pi)^{d}}  \sum_{k = 0}^\infty (1 + k(k+d-2))^s \int_0^\infty \|Y^r_k\|_2^2 \,r^{d-1} \, \mathrm{d}r
\]
since $\Delta Y_k^r = - k(k+d-2)Y_k^r$. The key point in the proof of Theorem \ref{t:roughdata} is to obtain the corresponding representation of $\widehat{\rho f}$ in terms of spherical harmonics.
\begin{lemma} \label{l:rhoradial}
Suppose $d \geq 2$ and $f$ is given by \eqref{e:radialexpansion}. Then
\[
\widehat{\rho f}(\xi,\tau) = \frac{2\pi|\mathbb{S}^{d-2}|}{|\xi|} \bigg( 1 - \frac{\tau^2}{|\xi|^2} \bigg)^{\frac{d-3}{2}}_+ \sum_{k=0}^\infty p_{d,k}(-\tfrac{\tau}{|\xi|}) Y_k^{|\xi|}(\xi')
\]
for each $(\xi,\tau) \in \mathbb{R}^{d+1}$ with $\xi \neq 0$.
\end{lemma}
The proof relies on the following classical theorem from harmonic analysis whose statement requires the introduction of the Legendre polynomial $p_{d,k}$ of degree $k$ in $d$ dimensions. We may define $p_{d,k}$ by the Rodrigues representation formula
\begin{equation*}
(1-t^2)^{\frac{d-3}{2}}p_{d,k}(t)=(-1)^k \frac{\Gamma(\frac{d-1}{2})}{2^k\Gamma(k+\frac{d-1}{2})} \frac{\mathrm{d}^k}{\mathrm{d}t^k}(1-t^2)^{k+\frac{d-3}{2}}
\end{equation*}
and we refer the reader to \cite{AH} for this definition and terminology.
\begin{theorem}[Funk--Hecke] \label{t:FH}
Let $d \geq 2$, $k \in \mathbb{N}_0$ and $Y_k$ be a spherical harmonic of degree $k$. Then
\begin{equation*}
\int_{\mathbb{S}^{d-1}} F(\omega \cdot \omega') Y_k(\omega') \, \mathrm{d}\sigma(\omega') = \zeta_k Y_k(\omega)
\end{equation*}
for any $\omega \in \mathbb{S}^{d-1}$ and any function $F \in L^1([-1,1],(1-\lambda^2)^{\frac{d-3}{2}})$. Here
\begin{equation*}
\zeta_k = |\mathbb{S}^{d-2}| \int_{-1}^1 F(\lambda) p_{d,k}(\lambda) (1-\lambda^2)^{\frac{d-3}{2}} \, \mathrm{d}\lambda.
\end{equation*}
\end{theorem}
We also suggest that the reader consults \cite{AH} for a treatment of the Funk--Hecke theorem.
\begin{proof}[Proof of Lemma \ref{l:rhoradial}]
As an immediate application of Theorem \ref{t:FH}, using \eqref{e:rhofhat} we obtain
\begin{align*}
\widehat{\rho f}(\xi,\tau) & = \frac{2\pi}{|\xi|} \sum_{k=0}^\infty \int_{\mathbb{S}^{d-1}} Y^{|\xi|}_k(v) \delta(\tfrac{\tau}{|\xi|} + \xi' \cdot v) \, \mathrm{d}\sigma(v) \\
& = \frac{2\pi}{|\xi|} \sum_{k=0}^\infty \zeta_k(\xi,\tau) Y^{|\xi|}_k(\xi')
\end{align*}
where
\begin{align*}
\zeta_k(\xi,\tau) & = |\mathbb{S}^{d-2}| \int_{-1}^1 \delta(\tfrac{\tau}{|\xi|} +\lambda) p_{d,k}(\lambda) (1-\lambda^2)^{\frac{d-3}{2}} \, \mathrm{d}\lambda.
\end{align*}
The claimed expression in the statement of Lemma \ref{l:rhoradial} follows.
\end{proof}
\begin{proof}[Proof of Theorem \ref{t:roughdata}]
Using Lemma \ref{l:rhoradial}, polar coordinates and orthogonality of $(Y^r_k)_{k \in \mathbb{N}_0}$ for each fixed $r > 0$,
\begin{align*}
\| \diffp\diffn  \rho f\|_{L^{2}}^2 & = \frac{|\mathbb{S}^{d-2}|^2}{(2\pi)^{d-1}} \sum_{k=0}^\infty \int_0^\infty \int_{-r}^r (r + |\tau|)^{2\beta_+} (r - |\tau|)^{2\beta_-} \times \\
& \qquad \qquad \bigg(1 - \frac{\tau^2}{r^2}\bigg)^{d-3} |p_{d,k}(-\tfrac{\tau}{r})|^2 \|Y_k^r\|_{L^2}^2 \, r^{d-3} \, \mathrm{d}\tau \mathrm{d}r
\end{align*}
and since $\beta_+ + \beta_- = \frac{1}{2}$, we have
\[
\| \diffp\diffn  \rho f\|_{L^{2}}^2 = \frac{2|\mathbb{S}^{d-2}|^2}{(2\pi)^{d-1}}\sum_{k=0}^\infty I_k \int_0^\infty \|Y_k^r\|_{L^2}^2 \, r^{d-1} \, \mathrm{d}r
\]
where
\[
I_k =  \int_{0}^1 |p_{d,k}(\lambda)|^2 (1+\lambda)^{d-3 + 2\beta_+} (1-\lambda)^{d-3 + 2\beta_-} \, \mathrm{d}\lambda.
\]
We now invoke the pointwise estimate
\begin{equation} \label{e:Legendrebound}
|p_{d,k}(\lambda)| \leq \min \{1, C_d k^{\frac{2-d}{2}}(1-\lambda^2)^{\frac{2-d}{2}}\}
\end{equation}
for each $|\lambda| < 1$ and $k \geq 1$, with explicit constant given by $C_d = 2^{d-2}\pi^{-1/2}\Gamma(\frac{d-1}{2})$. A proof of these estimates can be found, for example, in \cite{AH} (see the inequalities labelled (2.116) and (2.117) on pages 58--59).

It follows immediately from \eqref{e:Legendrebound} that
$
I_k \lesssim k^{2s}
$
for all $k \geq 1$, provided $s \in [-\frac{d-2}{2},0]$ and
$
\beta_- > - s - \frac{d-2}{2}.
$
Also, $p_{d,0} = 1$, so $I_0 \lesssim 1$ provided $\beta_- > - \frac{d-2}{2}$. It follows from the above that $\| \diffp\diffn  \rho f\|_{L^{2}} \lesssim \| (1 - \Delta)^{s/2} f\|_{L^{2}}$ for $s \in [-\frac{d-2}{2},0]$ and
$
\beta_- > - s - \frac{d-2}{2}.
$
\end{proof}

Using the above analysis as a key ingredient, we provide the following improvement to Theorem \ref{t:thm-sobs} for general $q$ and $r$ for initial data in $L^2_{\textrm{rad}(x)}$ and $L^2_{\textrm{rad}(x,v)}$. For simplicity we state the result with no scale for smoothing in the velocity variable; the interested reader may follow the above approach to generalise the result accordingly.
\begin{theorem} \label{t:radial000}
Let $d \geq 2$, $q,r\in [2,\infty)$ and suppose $\beta_+$, $\beta_-$ satisfy \eqref{e:scaling}. If
\[
\beta_- > \max\bigg\{ \frac{1}{q} + \frac{d-1}{2r} -\frac{3(d-1)}{4}, \frac{1-d}{2} \bigg\}
\]
then \eqref{e:sobs} holds for all $f \in L^2_{\emph{\textrm{rad}}(x)}$, and if
\[
\beta_- > \max\bigg\{ \frac{1}{q} + \frac{d-1}{r} - (d-1), \frac{1-d}{2}  \bigg\}
\]
then \eqref{e:sobs} holds for all $f \in L^2_{\emph{\textrm{rad}}(x,v)}$.
\end{theorem}
\begin{proof}
Our strategy is to follow the direct approach in Section \ref{section:direct}. In light of Lemma \ref{l:easybit} and Proposition \ref{p:abstract} (or, strictly speaking, the appropriate modification given we are restricting to $f \in L^2_{\textrm{rad}(x)}$), it suffices to prove \eqref{e:cdelta} for $f \in L^2_{\textrm{rad}(x)}$ such that $\widehat{f}(\cdot,v)$ is supported in $\mathfrak{A}_0$ for each $v \in \mathbb{S}^{d-1}$, and where $\eta = \max\{ \frac{1}{q} + \frac{d-1}{2r} -\frac{3(d-1)}{4}, \frac{1-d}{2}\}$. By \eqref{e:cdeltag}, it thus suffices to prove
\begin{equation} \label{e:CkL20}
\| \mathcal C_k  \rho f\|_{L^2}\lesssim  2^{\frac{2-d}{2}k}\|f\|_{L^2} \qquad (k \geq k_0)
\end{equation}
for such $f$. 

To see \eqref{e:CkL20}, we employ Lemma \ref{l:rhoradial} and polar coordinates to obtain 
\[
\| \mathcal C_k  \rho f\|_{L^2}^2 \lesssim \sum_{\ell = 0}^\infty \int_0^\infty \int_{-1}^1 \phi(r)^2 \psi(2^kr(1-\lambda))^2 (1-\lambda^2)^{d-3} |p_{d,\ell}(\lambda)|^2 \|Y^r_k\|_{L^2}^2 r^{d-2} \, \mathrm{d}\lambda \mathrm{d}r.
\]
Using the pointwise estimate $|p_{d,\ell}(\lambda)| \leq 1$ (see \eqref{e:Legendrebound}) we quickly obtain \eqref{e:CkL20} from this expression.

To prove the claimed estimate on $L^2_{\textrm{rad}(x,v)}$, we use Proposition \ref{p:radialStr} to improve upon \eqref{e:cdeltag} for $g$ which are radially symmetric in the spatial variable. Indeed, by the same argument used to prove \eqref{e:cdeltag} via Proposition \ref{p:Str}, it follows from Proposition \ref{p:radialStr} that for $q,r \in [2,\infty)$ with
$
\frac{1}{q} < (d-1)(\frac{1}{2} - \frac{1}{r})
$
we have
\[
\| \mathcal C_k  g\|_{L^q_tL_x^r}\lesssim  2^{-k/2}\|g\|_{L^2} \qquad (k \geq k_0)
\]
for all $g$ which are radially symmetric in the spatial variable. Hence, for all $\varepsilon > 0$ there exists $C_\varepsilon < \infty$ such that
\[
\| \mathcal C_k  g\|_{L^q_tL_x^r} \leq C_\varepsilon  2^{(\alpha^{**} + \varepsilon)k}\|g\|_{L^2} \qquad (k \geq k_0)
\]
where
\[
\alpha^{**} := \max\bigg\{\frac{1}{q} + \frac{d-1}{r} - \frac{d}{2},-\frac{1}{2}\bigg\}.
\]
It follows from \eqref{e:CkL20} that, for all $\varepsilon > 0$, \eqref{e:cdelta} holds (with an implicit constant depending on $\varepsilon$) for $\beta_- > \alpha^{**} + \frac{2-d}{2} + \varepsilon$. Proposition \ref{p:abstract} then implies \eqref{e:sobs} holds whenever $\beta_- > \alpha^{**} + \frac{2-d}{2}$, and this gives the claimed range of $\beta_-$ in the statement of Theorem \ref{t:radial000} for $f \in L^2_{\textrm{rad}(x,v)}$.
\end{proof}

\subsection{Sharp constants}
The duality principle in Theorem \ref{t:dualityprinciple} along with \eqref{e:mgamma} allows us to extract optimal constants for all cases of \eqref{e:sobs} when $(q,r) = (2,2)$, along with an identification of the class of extremisers.
\begin{theorem} \label{t:sharpgeneral}
Suppose $d \geq 2$ and $\beta_+,\beta_-$ satisfy \eqref{e:scaling} (i.e. $\beta_+ + \beta_- = \frac{1}{2}$) with $\beta_- \geq \frac{3-d}{4}$. Then the optimal constant in the estimate
\begin{equation} \label{e:sobsquared}
\| \diffp\diffn  \rho f\|_{L^2}^2\leq \mathbf{C}\|f\|_{L^2}^2
\end{equation}
for all initial data $f \in L^2$ is given by
\[
\mathbf{C} = 2\pi|\mathbb{S}^{d-2}| (d-2)^{2-d} (d-1-4\beta_-)^{\frac{d-1}{2} - 2\beta_-}(d-3+4\beta_-)^{\frac{d-3}{2} + 2\beta_-}
\]
for $\beta_- \in [\frac{3-d}{4},\frac{1}{4}]$, and
$
\mathbf{C} = 2\pi|\mathbb{S}^{d-2}|
$
for $\beta_- \in (\frac{1}{4},\infty)$. Furthermore, extremisers exist if and only if $(d,\beta_+,\beta_-) = (2,\frac{1}{4},\frac{1}{4})$, in which case $f$ is an extremiser if and only if
$$
\widehat{f}(\xi,v) = (|\xi|^2 - |\xi \cdot v|^2)^{1/4} \widehat{g}(\xi,-\xi\cdot v),
$$
where $g \in L^2$ is nonzero and $\widehat{g}$ is supported in $\mathfrak{C}$. In particular, nonzero functions in $L^2$ which are independent of the spherical variable are extremisers when $(d,\beta_+,\beta_-) = (2,\frac{1}{4},\frac{1}{4})$.
\end{theorem}
We note that when $d=2$ and $\beta_- = \frac{3-d}{4}$, the expression $0^0$ arises in the above formula for the optimal constant $\mathbf{C}$, and this should be interpreted as $0^0 = 1$ is each instance.
\begin{proof}
By \eqref{e:m-1} we may write
$m_{-1}(\xi,\tau)^2 = M(\frac{|\tau|}{|\xi|})$, where
\[
M(\lambda) = \tfrac{1}{2}|\mathbb{S}^{d-2}|   (1 +  \lambda)^{\frac{d-1}{2}-2\beta_-} (1 - \lambda)^{2\beta_- + \frac{d-3}{2}} \mathbf{1}_{[0,1]}(\lambda).
\]
Since $\sigma = 2\mu_{-1}$, it follows from Theorem \ref{t:dualityprinciple} that $\mathbf{C} = 4\pi \|M\|_\infty$. Elementary considerations may be used to show that this coincides with the claimed expression in the statement of Theorem \ref{t:sharpgeneral}.

Regarding extremisers, we observe that a necessary condition for existence is that $\|m_{-1}\|_\infty$ is attained on a set of positive measure in $\mathbb{R}^{d+1}$. However, it is clear that $\|M\|_\infty$ is attained at a single point if $(d,\beta_+,\beta_-) \neq (2,\frac{1}{4},\frac{1}{4})$, thus ruling out the possibility of extremisers.

When $(d,\beta_+,\beta_-) = (2,\frac{1}{4},\frac{1}{4})$ we have $\mathbf{C} = 4\pi$ and the function $M$ is identically equal to $\mathbf{1}_{[0,1]}$; hence extremisers exist. To give an identification of the class of extremisers, note that \eqref{e:sobsquared} holds if and only if $T = \mathcal{F}^{-1} m \widehat{\rho f}$ is a bounded operator $L^2 \to L^2$, with $m(\xi,\tau) = (|\xi|^2 - |\tau|^2)^{1/4}$, and  one can show that the class of extremisers for $T$ coincides with the image under $T^*$ of the class of extremisers for the dual inequality $T^* : L^2 \to L^2$ (see, for example, \cite{BJO}). By \eqref{e:mainpoint} it follows that $g$ is an extremiser for $T^*$ if and only if $g$ is an extremiser for the multiplier estimate
\[
\| \mathcal{F}^{-1}(m_{-1}\widehat{g})\|_2^2 \leq \|g\|_2^2.
\]
Since $m_{-1} = \mathbf{1}_\mathfrak{C}$ in the case $(d,\beta_+,\beta_-) = (2,\frac{1}{4},\frac{1}{4})$, it is necessary and sufficient for such $g$ to have Fourier support in $\mathfrak{C}$. Using \eqref{e:rhostarhat}, we see that $f$ is an extremiser if and only if
$$
\widehat{f}(\xi,v) = (|\xi|^2 - |\xi \cdot v|^2)^{1/4} \widehat{g}(\xi,-\xi\cdot v)
$$
for such $g$, as claimed.

Taking $\widehat{g}(\xi,\tau) = (|\xi|^2 - \tau^2)^{-1/4}g_0(|\xi|)$, where $g_0$ is a nonzero function such that $\int_0^\infty |g_0(r)|^2 r \,\mathrm{d}r < \infty$, then $g \in L^2$ with support in $\mathfrak{C}$. Moreover,
$$
(|\xi|^2 - |\xi \cdot v|^2)^{1/4} \widehat{g}(\xi,-\xi\cdot v)
$$
is independent of $v$, and hence such functions are amongst the class of extremisers.
\end{proof}

An inspection of the argument used to prove Theorem \ref{t:roughdata} when $s=0$ allows us to extract optimal constants and a characterisation of extremisers.
\begin{theorem} \label{t:sharpradial}
Suppose $d \geq 2$ and $\beta_+,\beta_-$ satisfy \eqref{e:scaling} (i.e. $\beta_+ + \beta_- = \frac{1}{2}$) with $\beta_- > \frac{2-d}{2}$. Then the optimal constant in the estimate
\[
\| \diffp\diffn  \rho f\|_{L^2}^2\leq \mathbf{C}_0\|f\|_{L^2}^2
\]
for initial data $f \in L^2_{\emph{\textrm{rad}}(x)}$ is given by
\[
\mathbf{C}_0 = 2^{2d-2}\pi \frac{|\mathbb{S}^{d-2}|^2}{|\mathbb{S}^{d-1}|} \mathrm{B}(\tfrac{1}{2};2\beta_+ + d-2,2\beta_- + d-2)
\]
and this is attained if and only if $f \in L^2_{\emph{\textrm{rad}}(x,v)}$.
\end{theorem}
Here, $\mathrm{B}(x;a,b) = \int_0^x \lambda^{a-1}(1-\lambda)^{b-1}\,\mathrm{d}\lambda$ denotes the incomplete beta function.
\begin{proof}
From the proof of Theorem \ref{t:roughdata} when $s=0$, it is clear that the step at which an inequality was made occurred when we used the bound $I_k \lesssim 1$ for all $k \geq 0$, where
\[
I_k = \int_{0}^1 |p_{d,k}(\lambda)|^2 (1+\lambda)^{d-3 + 2\beta_+} (1-\lambda)^{d-3 + 2\beta_-} \, \mathrm{d}\lambda.
\]
The uniform bound $|p_{d,k}(\lambda)| \leq 1 = p_{d,0}$ for all $k \geq 0$, $d \geq 2$ and $|\lambda| \leq 1$ gives that
\begin{equation} \label{e:Iktight}
I_k \leq 2^{2(d-2)}\mathrm{B}(\tfrac{1}{2};2\beta_+ + d-2,2\beta_- + d-2) = I_0
\end{equation}
for all $k \geq 0$, with equality if and only if $k=0$. This gives the claimed inequality in the statement of Theorem \ref{t:sharpradial}, and the optimality of the constant is clear by taking $f \in L^2_{\textrm{rad}(x,v)}$, for then all terms $Y^r_k$ are zero in the expansion \eqref{e:radialexpansion} for $k \geq 1$. Conversely, if $f$ is an extremiser then the fact that \eqref{e:Iktight} holds strictly for $k \geq 1$ forces $Y^r_k$ to vanish for almost all $r > 0$.
\end{proof}

\begin{acknowledgements}
This work was supported by the European Research Council grant number 307617 (Bennett), JSPS Research Activity Start-up no. 26887008 and JSPS Grant-in-Aid for Young Scientists A no. 16H05995 (Bez), and NRF Republic of Korea no. 2015R1A2A2A05000956 (Lee). The authors would also like to thank the anonymous referee for their helpful comments on the first draft of the paper.
\end{acknowledgements}

\end{document}